\documentclass[graybox, vecphys]{svmult}

\usepackage{type1cm}        
\usepackage{makeidx}         
\usepackage{graphicx}        
\usepackage{multicol}        
\usepackage[bottom]{footmisc}

\usepackage{newtxtext}       %
\usepackage{newtxmath}       

\newcommand{\R}{{\mathbb R}}

\newcommand{\N}{{\mathbb N}}

\newcommand{\mat}{\mathbf}  

\makeindex             
\usepackage{bibentry}
\nobibliography* 
\usepackage[resetlabels]{multibib}
\newcites{ES}{References}     

\usepackage{subcaption}

\begin{document}
	
\title*{Compressive Sensing and Neural Networks from a Statistical Learning Perspective}
\author{Arash Behboodi, Holger Rauhut and Ekkehard Schnoor}
\institute{
Arash Behboodi \at Institute for Theoretical Information Technology, RWTH Aachen University, Germany \\ \email{arash.behboodi@ti.rwth-aachen.de}
\and 
Holger Rauhut \at Chair for Mathematics of Information Processing, RWTH Aachen University, Germany \\
\email{rauhut@mathc.rwth-aachen.de}
\and 
Ekkehard Schnoor \at Chair for Mathematics of Information Processing, RWTH Aachen University, Germany \\ \email{schnoor@mathc.rwth-aachen.de}
}
%
%
\maketitle

\abstract*{Each chapter should be preceded by an abstract (no more than 200 words) that summarizes the content. The abstract will appear \textit{online} at \url{www.SpringerLink.com} and be available with unrestricted access. This allows unregistered users to read the abstract as a teaser for the complete chapter.
Please use the 'starred' version of the \texttt{abstract} command for typesetting the text of the online abstracts (cf. source file of this chapter template \texttt{abstract}) and include them with the source files of your manuscript. Use the plain \texttt{abstract} command if the abstract is also to appear in the printed version of the book.}


\abstract{
Various iterative reconstruction algorithms for inverse problems can be unfolded as neural networks. Empirically, this approach has often led to improved results, but
theoretical guarantees are still scarce. While some progress on generalization properties of neural networks have been made, great challenges remain. In this chapter, we discuss and combine these topics to present a generalization error analysis for a class of neural networks suitable for sparse reconstruction from few linear measurements. The hypothesis class considered is inspired by the classical iterative soft-thresholding algorithm (ISTA). The neural networks in this class are obtained by unfolding iterations of ISTA and learning some of the weights. 
Based on training samples, we aim at learning the optimal network parameters via empirical risk minimization and thereby the optimal network that reconstructs signals from their compressive linear measurements. In particular, we may learn a sparsity basis that is shared by all of the iterations/layers and thereby obtain a new approach for dictionary learning.
For this class of networks, we present a  generalization bound, which is based on bounding the Rademacher complexity of hypothesis classes consisting of such deep networks via Dudley's integral. 
Remarkably, under realistic conditions, the generalization error scales only logarithmically in the number of layers, and at most linear in number of measurements.
}

 \allowdisplaybreaks
\section{Introduction}
\label{ABHRESsec:1}
 
Learning representations of, or extracting features from, data is an important aspect of deep neural networks. In the past decade, this approach has led to impressive results and achieved state-of-the-art performances, e.g., for various classification tasks. However, due to the black-box nature of the end-to-end learning of neural networks, such features are usually abstract and difficult to interpret. On the other hand, algorithms such as the iterative soft-thresholding algorithm (ISTA) can be regarded as neural networks. Thus, with the help of modern deep learning software libraries, they can easily be implemented and optimized, such that the trained parameters can adapt to data sets of interest. When such algorithms are well understood, it can be possible to transfer results shown for the classical variant to their neural network variant and in this way increase our understanding of deep neural networks. 
A class of neural networks that we discuss in the present work aims at joint reconstruction and dictionary learning problem based on unfolding iterative soft-thresholding algorithm. Here, unfolding means that each step of an iterative algorithm constitutes a neural network layer whose parameters can be learned from data.

Here, the learned representation (a dictionary) is a very well-understood model in image and signal processing, which can be easily interpreted and visualized. As a practical application, one may think of reconstructing images from measurements taken by a medical imaging device. Instead of only trying to reconstruct the image, we would like to implicitly learn also a meaningful representation system which is adapted to the image class of interest, and leads to good generalization (e.g., when taking measurements of new patients). More generally, this is the approach of solving inverse problems in a data-driven way, e.g., by training neural networks 
\citeES{arridge2019solving, gottschling2020troublesome}.


{
The natural question arises how well these learned reconstruction methods work. We take the viewpoint of statistical learning theory and assume the data (signals, images etc.) to be generated independently by some unknown distribution. Generalization bounds give probabilistic estimates on the difference between the true error (with respect to the unknown distribution) and the empirical error for a hypothesis function. Thereby, such bounds predict how well a learned neural network performs on yet unseen data. By now, classical results bound the generalization error in terms of the VC-dimension or based on the Rademacher complexity \citeES{shalev-shwartz_understanding_2014,bartlett_rademacher_2002}. More recent methods include a compression approach \citeES{arora2018stronger} and a PAC-Bayesian approach \citeES{neyshabur_pac-bayesian_2018}. So far, generalization properties of neural networks have been studied mostly in the context of classification using feed-forward neural networks, see, e.g.,\ \citeES{bartlett_spectrally-normalized_2017,golowich_size-independent_2018,neyshabur_pac-bayesian_2018}. Especially, in the overparametrized scenario with more network parameters than samples which is common in deep learning, it is still a mystery why learned networks generalize very well, and present bounds cannot yet explain their success \citeES{zhang_understanding_2017,nagarajan2019uniform,jiang2019fantastic}, although some works attribute this to the so-called implicit bias of learning algorithms \citeES{neyshabur2017exploring,neyshabur_role_2019,neyshabur2014search,chou2020} such as the commonly used (stochastic) gradient descent. We will, however, not pursue this direction further in this chapter.}


The case studied here, a recurrent neural network used for a regression problem, 
has received less attention so far from the perspective of generalization.

Due to the weight sharing, this is a non-overparameterized network. However, it is straightforward to decouple the layers and thus obtain a network which is more similar to standard feed-forward neural networks.
Furthermore, we impose an orthogonality constraint on the dictionary, which in fact {constitutes the learned parameters of the network}. We derive generalization bounds for such thresholding networks with orthogonal dictionaries. {In order to upper bound the Rademacher complexity of the hypothesis class consisting of such deep networks, we apply a generalization of Talagrands contraction principle \citeES{ledoux_probability_2011} for
 vector-valued functions, which is typically not needed when considering real-valued hypothesis classes, e.g., with the ramp loss (applied to the margin) in a multiclass classification problem \citeES{bartlett_spectrally-normalized_2017}.
A similar idea for multi-class classification tasks has been tried in \citeES{neyshabur_role_2019}.
We further estimate the resulting expectation of the supremum of a certain Rademacher process via Dudley's integral (which in particular involves covering numbers) to upper bound the Rademacher complexity of hypothesis classes consisting of such deep networks.} 

{
Sample complexity of dictionary learning has been studied before in the literature \citeES{grsc10,vainsencher2011sample,sc14,gribonval2015sample,pmlr-v80-georgogiannis18a}. The authors in \citeES{vainsencher2011sample} also use a Rademacher complexity analysis for dictionary learning, but they aim at sparse representation of signals rather than reconstruction from compressed measurements and moreover, they do not use neural network structures.
Fundamental limits of dictionary learning from an information-theoretic perspective have been studied in \citeES{6952232, 7378975}. Unique about our perspective and different to the cited papers is our approach for determining the sample complexity based on learning a dictionary by training a neural network.}

{This chapter is structured as follows. In Section \ref{ABHRESsec:2}, we introduce learned soft iterative thresholding architecture, define the generalization error and review some of the related works. We discuss works on generalization bounds for deep neural networks in Section \ref{ABHRESsec:3} and introduce Rademacher complexity analysis. The main result of this chapter with detailed proofs is given in Section \ref{ABHRESsec:4}. Finally, we present the numerical results in Section \ref{ABHRESsec:5}. 
}

\subsection*{Notation}

Vectors $\vec{x} \in \R^N$ and matrices $\mat{A} \in \R^{n \times N}$ are denoted with bold letters, unlike scalars $\lambda \in \R$.  We will denote the spectral norm by 
$\| \mat{A} \|_{2 \to 2}$ and the Frobenius norm by $\| \mat{A} \|_F$. 
The $N \times m$ matrix $\mat{X}$ contains the data points, 
$\vec{x}_1, \dots, \vec{x}_m \in \R^N$ as its columns,
analogously $\mat{Y} \in \R^{n \times m}$ to collect the measurements $\vec{y}_1, \dots, \vec{y}_m \in \R^n$. 
As a short notation for indices we use $[m] := \{1, \dots, m\}$. 
To make the notation more compact, with a slight abuse of notation, for functions $f: \R^n \to \R^N$ we denote by 
$f(\mat{Y})$ the matrix whose $i$-th column is $f(\vec{y}_i)$. The unit ball of the $n$-dimensional normed space $\mathbb{R}^n$ is denoted by 
$B_{ \| \, \cdot \, \|}^n : = \{\vec{x} \in \mathbb{R}^n: \| \vec{x} \| \leq 1\}$. 
The covering number $\mathcal{N}\left(\mathcal{M}, d, \epsilon \right)$ of a metric space $(\mathcal{M}, d)$ at level $\epsilon$ is defined as the smallest number of balls of radius $\epsilon$ with respect to $d$ required to cover $\mathcal{M}$. 
When the metric is induced by some norm, we write $\mathcal{N} \left(\mathcal{M}, \| \, \cdot \, \|, \epsilon \right)$. 
We denote the $N$-dimensional orthogonal group by $O(N)$.


\section{Deep Learning and Inverse Problems}
\label{ABHRESsec:2}


{
During recent years, many works studied the application of neural networks in solving inverse problems (see for example \citeES{genzel_solving_2020,daras_intermediate_2021}). In this work, we focus on joint dictionary learning and sparse recovery using neural networks.} 
Compressive sensing using dictionaries has been studied before, but, in contrast to the scenario discussed here, typically using a fixed (and possibly even redundant) dictionary and a random measurement matrix \citeES{rauhut2008compressed}. 
The idea of interpreting {thresholded} gradient-steps of iterative algorithms such as ISTA \citeES{daubechies2004iterative} for sparse recovery as layers of neural networks is well-known since \citeES{gregor2010learning} and has since then been an active research topic, see, e.g., \citeES{Wu2020Sparse,liu2018alista, chen2018theoretical, kamilov2016learning, mousavi2015deep, xin2016maximal}. Thresholding networks fall into the larger class of proximal neural networks studied in \citeES{hahene19}. 
{The key aspect is to learn weight matrices for an unfolded version of ISTA.} Different works focus on different parametrization of the network for faster convergence and better reconstructions. Learning the dictionary can also be implicit in these works.  In this chapter, we consider algorithms that try to find a dictionary suitable for reconstruction. Some of the examples of these algorithms are the recently suggested {Learning ISTA (LISTA) \cite{gregor2010learning}}, Ada-LISTA \citeES{aberdam2020ada} and convolutional sparse coding \citeES{sreter2018learned} which learn efficient sparse and low-rank models \citeES{sprechmann2015learning}. 
Like many other related papers, such as ISTA-Net \citeES{zhang2018ista}, 
these methods are  mainly motivated by applications like inpainting \citeES{aberdam2020ada}. 

Instead of novel algorithmic aspects, our contribution is to conduct a generalization analysis for these algorithms, which to the best of our knowledge has not been addressed in the literature before in this particular setting. In this way, we connect this line of research with recent developments \citeES{golowich_size-independent_2018, bartlett_spectrally-normalized_2017} in the study of generalization of deep neural networks.

\vspace{1cm}

\subsection{Learned Iterative Soft-Thresholding}
\label{ABHRESsubsec:1}

Let us begin by recalling the well-known iterative soft-thresholding algorithm (ISTA) and how it can be interpreted as a neural network. 
Given a high-dimensional $s$-sparse signal $\vec{x} \in \R^N$ and a measurement matrix $\mat{A} \in \R^{n \times N}$ (i.e., taking $n$ linear measurements, with typically $s \ll N$), we would like to recover $\vec{x}$
from given $\vec{y} = \mat{A} \vec{x}$. 
Although this is an under-determined linear system of equations, under certain conditions on the 
signal (typically, as already mentioned above: sparsity) and on the (random) measurement matrix (null space property, restricted isometry property) the true signal $\vec{x}$ can be recovered \cite{foucart_mathematical_2013}. 
{A well-known reconstruction method is $\ell_1$-minimization, which consists in computing a minimizer of the convex optimization problem
\begin{equation}\label{ABHRES:l1:min}
\min_{\vec{x} \in \R^N} 
\frac{1}{2} \|\mat{A} \vec{x} -\vec{y}\|_2^2 + \| \vec{x} \|_1
\end{equation}
where $\|\vec{x}\|_1 = \sum_{\ell=1}^N |x_\ell|$ is the $\ell_1$-norm.
An actual algorithm for computing such minimizer} 
is {ISTA}
\citeES{daubechies2004iterative}, 
where we initialize $\vec{x}^{0} = \vec{0}$, and then recursively {compute}
\begin{eqnarray}
\vec{x}^{k+1} 
&=& S_{\tau \lambda}
    \left[\vec{x}^{k}  + \tau \mat{A}^\top(\vec{y} - \mat{A}\vec{x}^{k} )\right] \nonumber \\
&=& S_{\tau\lambda}\left[ \left(\mat{I}  - \tau \mat{A}^\top \mat{A}\right )\vec{x}^{k}   + \tau \mat{A}^\top \vec{y} \right],
\label{ABHRESeq:01}
\end{eqnarray}
 where $\lambda$ and $\tau$ are parameters of the algorithm, and $S_\lambda$ (applied entry-wise) is the shrinkage operator defined as
\begin{eqnarray}
    S_\lambda: \R  \to \R, \qquad  x & \mapsto     
    \begin{cases}
      0 & \mathrm{if} \;  |x| \; \leq \lambda, \\
      x - \lambda \operatorname{sign}(x)       & \mathrm{if} \; |x| \; > \lambda,
    \end{cases}
    \label{ES:def_Staulambda}
\end{eqnarray}
which can also be expressed in closed form as 
$S_{\lambda}(x) = \operatorname{sign}(x) \cdot \max(0,|x| - \lambda)$ for any $x \in \R$. {It is well-known, see e.g.\ \cite{daubechies2004iterative}, that $\vec{x}^{k}$ converges to a minimizer of \eqref{ABHRES:l1:min} under the condition} 
\begin{equation}\label{ABHRES:tau:cond}
\tau \|\mat{A}\|_{2 \to 2}^2 \leq 1.
\end{equation}


Note that \eqref{ABHRESeq:01} can be interpreted as a layer of a neural network with weight matrix $\mat{I}  - \tau \mat{A}^\top \mat{A}$, bias $\tau \mat{A}^\top \vec{y}$ and activation function $S_{\tau\lambda}$. 
As a side remark, let us observe that $S_\lambda$ can be written as the sum of two
\textit{rectified linear units} via 
$S_\lambda(x) = \operatorname{ReLU}(x-\lambda) - \operatorname{ReLU}(-x-\lambda)$. {Here,} $\operatorname{ReLU}(x) = \max(0,x)$ is one of the most popular activation functions used by deep learning practitioners, so that it is also often the default choice for theoretical investigations. While this may be regarded as a natural connection between ISTA and neural networks, we will make no more use of it, as it turned out to be convenient enough to work with $S_\lambda$ as the activation function itself.

This interpretation of ISTA as an \textit{unfolded neural network} has been studied for the first time in \citeES{gregor2010learning} {leading to the introduction of LISTA}. {Since then it has inspired research at the intersection of neural networks and inverse problems in recent years and many variants of neural network enhanced iterative thresholding algorithms have been proposed by now.}

Note that in the current form ISTA only takes the form of a neural network, but has no trainable parameters. To introduce trainable parameters, one may consider the following scenario. Namely, let us be given a class of signals $\vec{x} \in \R^N$ which are not necessarily sparse themselves, but sparsely representable with respect to a dictionary ${\mat{\Phi}}_o \in \R^{N \times N}$. In other words, for each $\vec{x}$ there is a sparse vector $\vec{z} \in \R^N$ such that $\vec{x} = {\mat{\Phi}}_o \vec{z}$. The dictionary ${\mat{\Phi}}_o$ is assumed to be unknown. 
It is possible to extend to overcomplete dictionaries, but we will stick to bases for the sake of simplicity.



We would like to learn a dictionary suitable for
{sparse reconstruction} 
from a training sequence 
$\mathcal{S} = \left((\vec{x}_i,\vec{y}_i)\right)_{i = 1, \dots, m}$ with i.i.d.\ samples drawn from an (unknown) distribution $\mathcal{D}$. 
Formally, this is a distribution over the $\vec{x}_i$, and then the corresponding measurements $\vec{y}_i$ are given by  $\vec{y}_i = \mat{A}\vec{x}_i$, with $ \mat{A}$ being fixed. {We assume that the signals $\vec{x}$ in the class are bounded by a value, say $B_{\operatorname{in}}$, in the $\ell_2$-norm.} 

While taking the measurements $\vec{y} = \mat{A}\vec{x} =: \operatorname{enc}_\mat{A} (\vec{x})$
may be interpreted as \textit{encoding} the signal $\vec{x}$ into $\vec{y}$, corresponding to a shallow, one-layer linear neural network (which is deterministic, when the measurement matrix $\mat{A}$ is considered to be fixed), the \textit{decoder} is based on the unfolded version of the iterative soft thresholding algorithm (ISTA) with $L$ iterations as follows. 
For a fixed stepsize $\tau > 0$, and a fixed $\lambda>0$,
the first layer is defined by $f_1(\vec{y}) =S_{\tau\lambda}(\tau (\mat{A} \mat{\Phi})^\top \vec{y})$.
For the  iteration (or layer, respectively) $l>1$, the output is given by
\begin{eqnarray}
f_l(\vec{z}) \label{ABHRES:def:layer}
&=& S_{\tau \lambda}
    \left[\vec{z} + \tau (\mat{A} \mat{\Phi})^\top(\vec{y}-(\mat{A} \mat{\Phi})\vec{z})\right] \\
&=& S_{\tau\lambda}\left[ \left(\mat{I}  - \tau \mat{\Phi}^\top  \mat{A}^\top \mat{A}\mat{\Phi} \right )\vec{z}          + \tau (\mat{A}\mat{\Phi} )^\top \vec{y} \right], \notag
\end{eqnarray}
which again can be interpreted as a layer of a neural network with weight matrix $\mat{I}  - \tau \mat{\Phi}^\top  \mat{A}^\top \mat{A}\mat{\Phi}$, bias $\tau (\mat{A}\mat{\Phi} )^\top \vec{y}$ and activation function $S_{\tau\lambda}$, where the trainable parameters are the entries of $\mat{\Phi}$.
Note that for $l>1$, all $f_l$ coincide as functions on $\R^N$. The index then refers to the iteration step or layer of the neural network, respectively.
Then we denote the concatenation of $l$ such layers as $f_\mat{\Phi}^l$, 
i.e., for $\mat{\Phi}$ in every layer and given by
\begin{equation}
f_\mat{\Phi}^L(\vec{y}) = f_L\circ f_{L-1}\dots \circ f_1(\vec{y}),
\label{ES:eq:thresholding_networks}
\end{equation}
{Note that, strictly speaking, the vector $\vec{y}$ will also be an input to the subsequent layers $f_2$, $f_3$ etc., but to simplify the notation, we do not write it explicitely after each layer. This point will not be of major importance for our derivations throughout this chapter. }

For an actual reconstruction we need to apply the dictionary $\mat{\Phi}$ again after the final layer. This means, a decoder (for a fixed number of layers $L$) is a neural network with shared weights
\begin{equation*}
\operatorname{dec}_\mat{\Phi}^L(\vec{y}) 
= 
\mat{\Phi} f_L\circ f_{L-1}\dots \circ f_1(\vec{y})
= 
\mat{\Phi} f_\mat{\Phi}^L(\vec{y}) .
\end{equation*}
For technical reasons which will become apparent later in the proofs in Section 3, we will add an additional function $\sigma$ after the final layer. Different choices are possible here; we consider the choice
\begin{eqnarray}
     \sigma: \R^N\to\R^N, \qquad  \vec{x} & \mapsto     
    \begin{cases}
      \vec{x} & \mathrm{if} \;  \|\vec{x}\|_2 \leq B_{\operatorname{out}}, \\
      B_{\operatorname{out}} \frac{\vec{x}}{\|\vec{x}\|_2} & 
            \mathrm{if} \;  \|\vec{x}\|_2 > B_{\operatorname{out}},
    \end{cases}
\label{ABHRES:Bout}
\end{eqnarray}
with some fixed constant $B_{\operatorname{out}}$. Obviously, this ensures 
$\| \sigma(\vec{x}) \|_2 \leq B_{\operatorname{out}}$. 
Furthermore, note that $\sigma$ is norm-contractive and $1$-Lipschitz, i.e., 
\begin{equation}\label{ABHRES:sigma:bound}
\| \sigma(\vec{x}) \|_2 \leq \| \vec{x} \|_2
\qquad \mathrm{and} \qquad 
\| \sigma(\vec{x}_1) - \sigma(\vec{x}_2)\|_2 
\leq 
\|\vec{x}_1 - \vec{x}_2\|_2
\end{equation}
for any $\vec{x}$ and  $\vec{x}_1,  \vec{x}_2 \in \R^N$. The role of $\sigma$ is to push the output of the network inside the $\ell_2$-ball of radius $B_{\operatorname{out}}$, which in many applications is approximately known. The prior knowledge about the range of the output (boundedness) can improve the reconstruction performance and generalization \cite{Wu2020Sparse}. {The constant $B_{\operatorname{out}}$ may} be simply chosen to be equal to $B_{\operatorname{in}}$. 

To formulate this as a statistical learning problem, we will formally introduce a hypothesis class and a loss function. The hypothesis set consists of all functions that can be expressed as $L$-step soft thresholding, where the dictionary matrix $\mat{\Phi}$ parameterizes the hypothesis class, and with an additional $\sigma$ after the final layer added. That is,
\begin{equation}
\mathcal{H}_1^L 
= 
\{\sigma \circ h:\R^n\to \R^N: h(\vec{y}) 
= \mat{\Phi} f_\mat{\Phi}^L(\vec{y}), \mat{\Phi} \in O(N)\}.
\label{ES:eq:hypothesis_class:H1}
\end{equation}
The assumption that $\mat{\Phi}$ ranges over the orthogonal group $O(N)$ and is shared across the layers leads to a recurrent neural network with a moderate number of weights. Using weight sharing enables a straightforward interpretation of learning a dictionary for reconstruction.
Much more general scenarios are discussed later, including models without weight-sharing (or different degrees thereof), and models where also the threshold $\lambda$ and the stepsize $\tau$ may be trainable, and even be altered from layer to layer.
 
Based on the training samples $\mathcal{S}$ and given the hypothesis space $\mathcal{H}_1^L$,
a learning algorithm yields a function $h_\mathcal{S} \in \mathcal{H}_1^L $ that aims at reconstructing 
$\vec{x}$ 
from the measurements $\vec{y}=\mat{A} \vec{x}$. The empirical loss of a hypothesis $h$ is the reconstruction error on the training sequence,
i.e., the difference between $\vec{x}_i$ and $\hat{\vec{x}}_i =  h(\vec{y}_i)$, that is
\[
\hat{\mathcal{L}}{(h)} = \frac{1}{m}\sum_{j=1}^m \ell(h,\vec{x}_j,\vec{y}_j).
\]
Different choices for the loss function $\ell$ to measure the reconstruction error are possible. A typical choice is the 
\textit{mean squared error} (MSE)
\begin{equation}
\ell_{\operatorname{MSE}}(h,\vec{x},\vec{y}) 
= 
\|h(\vec{y}) - \vec{x}\|_2^2.
\label{ABHRESeq:mse_loss}
\end{equation}
The true loss, i.e., the risk of a hypothesis $h$, is accordingly defined as follows by
\[
\mathcal{L}(h) = \mathbb{E}_{\vec{x},\vec{y}\sim\mathcal{D}}\left(\ell(h,\vec{x},\vec{y})\right).
\]
The generalization error of the hypothesis $h_{\mathcal{S}}$ based on the training samples $\mathcal{S}$ is given as the difference between the empirical loss 
and the true loss,
\[
\operatorname{GE}(h_\mathcal{S}) = \left|\hat{\mathcal{L}}(h_\mathcal{S}) - \mathcal{L}(h_\mathcal{S})\right|.
\]
Note that some references denote the true loss $\mathcal{L}(h_\mathcal{S})$ as the generalization error. 
However, the above definition is more convenient for our purposes.

This motivating example explains how iterative reconstruction algorithms such as ISTA can be unfolded as a neural network, which is then trained on some training data. 
By this transformation into a machine learning problem, this raises the question of generalization, i.e., how well the trained decoder works on unseen data (from the same distribution of interest).
We will return to the problem of bounding the generalization error in the third section, after giving a more detailed overview on LISTA and its variants in the remainder of this section, and discussing different approaches to and challenges of generalization in deep learning in the next section.

\subsection{Variants of LISTA}
\label{ABHRESsubsec:2}

Before moving to the generalization analysis, we review some of the variants of LISTA algorithms. The original paper \citeES{gregor2010learning} focuses on sparse coding applications where a sparse representation of data needs to be learned. Since $\ell_1$-based methods are slow and do not scale to larger datasets, the authors propose a time unfolded version of ISTA with fixed number of iterations. The first layer of the network is simply given by
\[
\vec{x}^{0} = S_{\lambda}
    \left[\mat{W}\vec{y} \right].
\]
The next layers are formulated as
\[
\vec{x}^{t} = S_{\lambda}
    \left[\mat{W}\vec{y} + \mat{S}\vec{x}^{t-1} \right].
\]
LISTA learns $(\lambda,\mat{W},\mat{S})$ by back-propagation. {These parameters can be  common or different across layers. }

{Following LISTA, many works explored the similar idea of unfolding iterative thresholding algorithms \citeES{Wu2020Sparse,yang_deep_2016,liu2018alista,zhang2018ista,sprechmann2015learning,xin2016maximal}. } For example, the authors of \citeES{Wu2020Sparse} address two problems of ISTA type methods. First, the convergence of LISTA requires higher thresholds in the shrinkage operator to produce sparse vectors.  This comes at the cost of shrinkage of the output with respect to the original vector. The authors in 
\citeES{Wu2020Sparse} introduce a gain gate to increase output values and compensate this effect. They additionally introduce an overshoot gate that tries to improve the convergence by learning to boost the estimated vector close to the ground truth.


As mentioned before, it is possible to change further the structure of LISTA and {possibly} improve the performance. For example, in {analytic LISTA (ALISTA)} \citeES{liu2018alista},  only thresholds and step-size parameters are learned. {The update rule 
\eqref{ABHRES:def:layer} is modified into
$f_l(\vec{z}) = S_{\lambda^{(l)}}
     \left[\vec{z} + \tau^{(l)} \mat{W}^\top(\vec{y} - \mat{A}\vec{x}^{l} )\right]$, where the matrix $\mat{W}$ is chosen without using data, namely as a minimizer of the coherence with respect to $\mat{A}$.}
Instead of learning the same step sizes and thresholds for all the samples as in ALISTA, these parameters are updated based on the output of the previous layer in neurally augmented ALISTA \citeES{behrens_neurally_2021}.

\section{Generalization of Deep Neural Networks}
\label{ABHRESsec:3}

The generalization error of machine learning algorithms {is the gap between their performance averaged over the samples of training data and the expected performance computed using the actual distribution.} In this chapter, we define the generalization error as {the absolute difference between these two losses.} 

A machine learning algorithm $\mathcal{A}$ returns a function $h:\mathcal{X}\to\mathcal{Y}$ from a set of choice functions called hypothesis class $\mathcal{H}$ based on the training data defined as $m$ i.i.d. samples  $\vec{z}_i = (\vec{x}_i,{y}_i)$ from a  distribution $\mathcal{D}$ on $\mathcal{X} \times  \mathcal{Y}$.
If the hypothesis class is \textit{large}, it may contain complex enough functions that match the training data perfectly with zero training error. These functions, however, do not necessarily generalize well to {new, yet unseen data (or the test data in experiments)}. Statistical learning theory aims at bounding the generalization error in terms of the complexity of hypothesis class and training set size. There are different notions of complexity available in the literature such as Vapnik-Chernovekis (VC) dimension  \citeES{vapnik_nature_2000, vapnik_uniform_2015}, Rademacher complexity \citeES{koltchinskii_rademacher_2001, bartlett_rademacher_2002}, stability \citeES{shalev-shwartz_learnability_2010,rakhlin_stability_2005} and robustness \citeES{xu_robustness_2012}.
{In this chapter, we focus on 
the Rademacher complexity framework for bounding the generalization error.}


\subsection{Rademacher Complexity Analysis}
\label{ESsubsec:2:1} 


In order to bound the generalization error, we use the Rademacher complexity. Consider a class $\mathcal{G}$ of real-valued functions $g$.
The empirical Rademacher complexity is defined as 
\begin{equation}
\mathcal{R}_\mathcal{S}(\mathcal{G}) 
= 
\mathbb{E}_{\epsilon} \sup_{h \in \mathcal{G}} \frac{1}{m} \sum_{i=1}^m \epsilon_i g(\vec{x}_i),
\end{equation}
where $\epsilon$ is a Rademacher vector, i.e., a vector of independent Rademacher variables $\epsilon_i$, $i=1,\hdots,m$, taking the values $\pm 1$ with equal probability.
The Rademacher complexity is then given as $\mathcal{R}_m(\mathcal{G})= \mathbb{E}_{\mathcal{S} \sim \mathcal{D}^m} \mathcal{R}_\mathcal{S}(\mathcal{G})$. We will exclusively work with the empirical Rademacher complexity. The Rademacher complexity 
provides a complexity measure that can bound the generalization error. Suppose that the training samples are given by 
$\mathcal{S} = (z_1,\hdots,z_m)$, 
{where $z_i \in \mathcal{Z}=\mathcal{X}\times\mathcal{Y}$}. 
The hypothesis class $\mathcal{H}$ consists of function $h:\mathcal{X}\to\mathcal{Y}$.
{Consider a loss function $\ell:\mathcal{H}\times Z\to \R$.}
The empirical loss of a function $h$ is defined by
\[
\hat{\mathcal{L}}(h) = \frac{1}{m}\sum_{j=1}^m \ell(h,z_j).
\]
This is the performance of $h$ on the training data. We can write the true loss of $h$ as
\[
\mathcal{L}(h) = \mathbb{E}_{z\sim\mathcal{D}}\left(\ell(h,z)\right).
\]
 Given a loss function $\ell$ and a hypothesis class $\mathcal{H}$, we are interested in the Rademacher complexity of the class 
$\mathcal{G} = \ell \circ \mathcal{H} = \{ g(z) = \ell\circ h(z): h \in \mathcal{H}\}$. We rely on the following theorem which bounds the generalization error in terms of the empirical Rademacher complexity.

\begin{theorem}[{\citeES[Theorem 26.5]{shalev-shwartz_understanding_2014}}]
\label{ES:thm:ge_vs_rademacher}
{Let $\mathcal{H}$ be a family of functions,  $\mathcal{S}$ the training set drawn from $\mathcal{D}^m $, and $\ell$ a real-valued bounded loss function satisfying $|\ell(h,z)| \leq c$ for all $h\in \mathcal{H}, z\in Z$. Then, for $\delta \in (0,1)$, with probability at least $1-\delta$  
we have, for all $h \in \mathcal{S}$,}
\begin{equation}
\mathcal{L}(h) 
\leq 
\hat{\mathcal{L}}(h) + 2\mathcal{R}_\mathcal{S}(\ell \circ\mathcal{H}) 
+ 4c \sqrt{\frac{2\log(4/\delta)}{m}}. 
\label{eq:ge_vs_rademacher:1}
\end{equation}
\end{theorem} 

For real-valued functions $h$, i.e. when $\mathcal{Y}=\R$, the Rademacher complexity of $\ell\circ\mathcal{H}$ can be bounded using the so-called contraction lemma \cite[Lemma~26.9]{shalev-shwartz_understanding_2014}.

\begin{lemma}[Contraction lemma]
Let $\mathcal{S}$ be the training sequence and the functions $f_i$ be $K$-Lipschitz from $\R$ to  $\R$ for $i\in[m]$. Then, {for a class of real valued functions $\mathcal{H}$}, we have
\begin{equation}
       \mathbb{E}_{\epsilon} \sup_{h\in\mathcal{H}} \sum_{i=1}^m \epsilon_i f_i \circ h(\vec{x}_i)
\leq   K \mathbb{E}_{\epsilon} \sup_{h \in \mathcal{H}} \sum_{i=1}^m  \epsilon_{i} h(\vec{x}_i).
\end{equation}
\label{eq:contraction}
\end{lemma}

With the contraction lemma, we can remove the loss function and work only with the hypothesis class. 
\subsection{Generalization bounds for Deep Neural Networks}
\label{ESsubsec:5}

{
Many recent works aim at  explaining the excellent generalization properties of deep neural networks. In order to provide a brief review of this body of literature, we consider an $L$-layer neural network
\[
f_{\mat{W}_1, \hdots, \mat{W}_L} (\vec{y}) 
= 
\sigma(\mat{W}_L \cdot \sigma( \cdots \sigma(\mat{W}_1 \vec{y}) \cdots ) )
\]
with weight matrices $\mat{W}_j \in \mathbb{R}^{n_{j-1}\times n_j}$, $j = 1,\hdots,L$ ($n_0 = n$), and an elementwise activation function $\sigma : \mathbb{R} \to \mathbb{R}$}. {All the works to be mentioned below consider the matrices $\mat{W}_j$ as free parameters of the hypothesis class, hence, they aim at an overparameterized setting. Moreover, they consider classification problems. In contrast, our work considers regression problems and networks with shared weights, leading to a non-overparametrized setting.} 

{
The Rademacher complexity was used in \citeES{bartlett_spectrally-normalized_2017} to obtain norm-based generalization error {bounds for the probability of misclassification via the argmax of a neural network in a multi-class problem with $K$ classes.
The margin-type bound in \citeES{bartlett_spectrally-normalized_2017} states that, with probability at least $1-\delta$ over the i.i.d.\ samples $(\vec{x}_i,\vec{y}_i) \in \mathbb{R}^N \times [K]$,}
\begin{align*}
&\mathbb{P}(\operatorname{argmax} f_{\mat{W}_1, \hdots, \mat{W}_L}(\vec{x}) \neq \vec{y}) \leq \widehat{\mathcal{R}}_\gamma(f_{\mat{W}_1,\hdots, \mat{W}_L}) \\
& + C \frac{\|\mat{X}\|_F \log(\max_j n_j)}{\gamma m}
\left(
    \sum_{i=\ell}^L \left(
    \frac{\|\mat{W}_\ell^T\|_{2,1}}{\|\mat{W}_\ell\|_{2 \to 2}}\right)^{2/3}
    \right)^{3/2}
    \prod_{\ell=1}^L (\rho \|\mat{W}_\ell\|_{2 \to 2})
 +  \sqrt{\frac{C \log(1/\delta)}{m}},
\end{align*}
where $\gamma > 0$ is a suitable parameter, $\rho>0$ the Lipschitz constant of $\sigma$ and $\|\mat{W}_\ell^T\|_{2,1}$ is the sum of $\ell_2$-norm of rows of $\mat{W}_\ell^T$. Furthermore, the empirical margin-type risk is defined as
\[
\widehat{\mathcal{R}}_\gamma(f) = \frac{1}{m} \sum_{i=1}
^m \mathbf{1}[f(\vec{x}_i)_{\vec{y}_i} \leq \gamma + \max_{j \neq \vec{y}_i} f(\vec{x}_i)_j]. 
\]
Noting that in general $\|\mat{W}_\ell^T\|_{2,1} = \|\mat{W}_\ell\|_{2 \to 1} \leq \sqrt{n_\ell} \|\mat{W}_\ell\|_{2 \to 2}$, and assuming that all data points $\vec{x}_i$ are bounded in $\ell_2$, i.e., $\|\mat{X}\|_F \leq \sqrt{m} B$, the first term in the second line of the bound can be estimated by 
\[
L^{3/2} \max_j \sqrt{n_j} \log(n_j) \left(\rho \max_\ell \|\mat{W}_\ell\|_{2 \to 2}\right)^L \frac{B}{\gamma \sqrt{m}}.
\]
A similar, but slightly worse, norm based bound was obtained  \citeES{neyshabur_pac-bayesian_2018} using a PAC Bayesian approach, which leads to a completely different analysis.} 


A bound with {potentially} better dimension dependence was obtained in \citeES{golowich_size-independent_2018}. As an example, the result of \citeES{bartlett_spectrally-normalized_2017} can be improved {to a generalization error bound scaling in the following way}, for any $p\geq 1$,
\begin{equation}
    \tilde{O}\left(\prod_{\ell=1}^L M(\ell)
    \min\left\{
    \frac{\log\left(\frac{1}{\Gamma}\prod_{\ell=1}^L M_p(\ell)
    \right)^{\frac{1}{p+\frac{2}{3}}}}{{m}^{\frac{1}{2+3p}}},
    \sqrt{\frac{L^3}{m}}
    \right\}
    \right),
\end{equation}
where $M(\ell)$ and $M_p(\ell)$ are respectively upper bounds on $\|\mat{W}_\ell\|_{2\to 2}$ and $p$-Schatten norm of $\mat{W}_\ell$, and $\Gamma$ is a lower bound on $\prod_{\ell=1}^L \|\mat{W}_\ell\|_{2 \to 2}$. This result can remove the dependence on $L$ if the norms are well behaved. Note that this comes at the price of a worse sample efficiency. For instance, for $p=1$, the dimension free bound scales as $m^{1/5}$ in contrast with $m^{1/2}$ {(however, note that a minimum with $\sqrt{L^3/m}$ is taken)}.

{Instead of continuing the discussion on existing generalization bounds for deep networks, we invite the interested reader to refer to \citeES{jiang2019fantastic} for a detailed experimental account and \citeES{nagarajan2019uniform} for some shortcomings of existing bounds, for example, they tend to grow with training data size.}

{In the remainder of this chapter, we will show a generalization error bound for regression (reconstruction) with the introduced thresholding networks which shows linear dimension dependence and linear dependence on the number of layers, see Theorem~\ref{ABHRES:thm:main_result}. Our proof uses
different techniques than in
the works mentioned above.
In fact, it is not straightforward to apply those techniques due to the weight sharing between different layers in our case.}

\section{Generalization of Deep Thresholding Networks}
\label{ABHRESsec:4}

In this section, we return to the original problem layed out already in the beginning of this chapter.
{We will prove the following result on the generalization error of the class of neural networks $\mathcal{H}_1$ introduced in Section~\ref{ABHRESsubsec:1} with a learned orthogonal dictionary. 
We state our theorem here under the simplifying but reasonable assumption that $\tau \|\mat{A}\|_2^2 \leq 1$, see \eqref{ABHRES:tau:cond}. A more general version of the result will be presented in Section~\ref{ABHRES:subsec:mainresult}. We continue to use the notation introduced in Section~\ref{ABHRESsubsec:1}.} 

{
\begin{theorem}\label{ABHRES:thm:main_result}
Consider the hypothesis space $\mathcal{H}_1^L$ defined in \eqref{ES:eq:hypothesis_class:H1} and assume the samples $\vec{x}_i$, $i= 1,\hdots, m$, to be drawn i.i.d.\ at random according to some (unknown) distribution such that $\|\vec{x}_i\|_2 \leq B_{\operatorname{in}}$ almost surely with $B_{\operatorname{in}} = B_{\operatorname{out}}$ in \eqref{ABHRES:Bout}. 
Let  $\vec{y}_i = \mat{A}\vec{x}_i$ and assume that
$\tau \|\mat{A}\|_{2 \to 2}^2 \leq 1$.
Then with probability at least $1-\delta$, for all $h \in \mathcal{H}_1^L$, the generalization error is bounded as
\begin{eqnarray}
\mathcal{L}(h) 
 \leq 
\hat{\mathcal{L}}(h) 
& + & 8 B_{\operatorname{out}} \sqrt{\frac{Nn \log(2 + 8L(L+3))}{m}}
+ 8 B_{\operatorname{out}} \frac{N\sqrt{\log(e + 8eL)}}{\sqrt{m}}
 \notag\\
& & + B_{\operatorname{out}} \sqrt{\frac{128\log(4/\delta)}{m}}. 
\label{ABHRES:eq:main:bound}
\end{eqnarray}
\end{theorem}}

Further details and a slightly simplified bound for $L \geq 2$ can be found in Corollary \ref{ES:corollary:main_result} and the discussion thereafter.


{
Of course, the idea is to choose an $h$ that minimizes the empirical loss $\hat{\mathcal{L}}(h)$, i.e., the first term on the right hand side of \eqref{ABHRES:eq:main:bound}, but in principle any $h$ (computed by some algorithm) can be inserted into this bound. Since the samples are available, both $\hat{\mathcal{L}}(h)$ and the other terms can be computed (assuming $B_{\operatorname{in}}$ is known), so that the theorem allows to provide a concrete bound of the true risk $\mathcal{L}(h)$.
Roughly speaking, i.e., ignoring constants, the generalization error can be bounded as 
\begin{equation}
    |\mathcal{L}(h) - \hat{\mathcal{L}}(h)| 
\lesssim 
\sqrt{\frac{Nn \log(L) +N^2 \log(L)}{m}}.
\label{eq:general_error_rough}
\end{equation}

In other words, once the number of training samples scales like 
$ m \sim (Nn + N^2) \log(L)$, 
the generalization error is guaranteed to be small with high probability.

Remarkably, the number $L$ of layers only enters logarithmically, while some of the previously available bounds for deep neural networks (in the context of classification, however) scale even only exponentially with $L$ (at least in many interesting settings). 
}


{The remainder of this section is devoted to the proof of the above statement. We} will use the approach based on
the Rademacher complexity as described in Section~\ref{ESsubsec:2:1}, in particular 
Theorem~\ref{ES:thm:ge_vs_rademacher}. Hence, we need to estimate the Rademacher complexity
\begin{equation}
        \mathcal{R}_m (\ell \circ \mathcal{H}^L_1) 
=   \mathbb{E} \sup_{h \in \mathcal{H}^L_1} \frac{1}{m} 
    \sum_{i = 1}^m \epsilon_i \|\sigma( h(\vec{y}_i) - \vec{x}_i) \|_2.
\end{equation}
{As explained in Section~\ref{ESsubsec:2:1}, the so-called contraction principle is often applied in such situations}. However, since we are dealing with a hypothesis class of vector-valued functions, it is not applicable in its standard form. The following result \citeES[Corollary 4]{maurer_vector-contraction_2016} is a generalization to this situation, and it is a crucial tool for our proof.

\begin{lemma}\label{ES:lem:Maurer}
{Suppose that $\mathcal{H}$ is a set of functions $h: \mathcal{X} \to \R^N$ and that $f: \R^N \to \R^N$ is $K$-Lipschitz.} Let $\mathcal{S}= (\vec{x}_i)_{i \in [m]}$ be the training sequence. Then
\begin{equation}
       \mathbb{E} \sup_{h\in\mathcal{H}} \sum_{i=1}^m \epsilon_i f \circ h(\vec{x}_i)
\leq   \sqrt{2} K \mathbb{E} \sup_{h \in \mathcal{H}} \sum_{i=1}^m \sum_{k=1}^N \epsilon_{ik} h_k(\vec{x}_i),
\end{equation}
{where $(\epsilon_i)$ and $(\epsilon_{ik})$ are both Rademacher sequences.}
\end{lemma}
As both the $\ell_2$-norm and the function $\sigma$ (the latter by assumption) are $1$-Lipschitz, applying Lemma \ref{ES:lem:Maurer} yields
\begin{equation}
        \mathcal{R}_m (\ell \circ \mathcal{H}^L) 
\leq    \sqrt{2} \mathbb{E} \sup_{h \in \mathcal{H}^L} \frac{1}{m} 
        \sum_{i = 1}^m \sum_{k = 1}^N \epsilon_{ik} h_k(\vec{x}_i).
        \label{ES:rademacher_2}
\end{equation}
%
%
In order to derive a bound for the Rademacher complexity, we use chaining techniques.
Roughly speaking, this refers to bounding the expectation of a stochastic process by geometric properties of its index set (covering numbers at different scales), equipped with an appropriate norm (or metric).
We briefly provide the necessary results in the next section, for a more detailed introduction to the topic, we refer the reader to \citeES{ledoux_probability_2011,talagrand2014upper}.

\subsection{Boundedness: Assumptions and Results}

For technical reasons that will become apparent, we will introduce a separate dictionary for the linear transformation after the very final layer and consider the enlarged hypothesis class
\begin{equation}
\mathcal{H}_2^L 
= 
\{\sigma \circ h:\R^n\to \R^N: h(\vec{y}) 
= \mat{\Psi} f_\mat{\Phi}^L(\vec{y}), 
\mat{\Psi}, \mat{\Phi} \in O(N)\}.
\label{ES:eq:hypothesis_class:H2}    
\end{equation}

In order to apply Theorem \ref{ES:thm:ge_vs_rademacher}, the loss function needs to be bounded. Therefore, and as commonly done in the machine learning literature, we assume {(as already mentioned)} that the input is bounded in the $\ell_2$-norm by some constant $B_{\operatorname{in}}$, i.e.,
\begin{equation}
\left \| \vec{x} \right \|_2 \leq B_{\operatorname{in}}.
\label{ES:eq:B_in}
\end{equation}
Furthermore, let us recall {from \eqref{ABHRES:sigma:bound}} that the function $\sigma$
is bounded by $B_{\operatorname{out}}$. In particular, this means that every
$h \in \mathcal{H}^L_2$ (analogously for $\mathcal{H}^L_1$) is also bounded by
\begin{equation}
\left \| h (\vec{y}) \right \|_2
=
\left \| \sigma \left(\mat{\Psi} f_\mat{\Phi}^L(\vec{y})\right) \right \|_2
\leq 
B_{\operatorname{out}}
\label{ES:eq:hypothesis_B_out}
\end{equation}
independently of $\mat{\Psi}$ and $\mat{\Phi}$. 
By passing to the matrix notation (i.e., considering the matrix $\mat{Y}$ collecting all measurements, instead of a single measurement $\vec{y}$), we obtain the similar estimate
\begin{equation}
\left \| h (\mat{Y}) \right \|_F
\leq 
\sqrt{m} B_{\operatorname{out}}
\label{ES:eq:output_matrix}
\end{equation}
{where the additional term of $\sqrt{m}$ takes the number of training points into account.}
By combining \eqref{ES:eq:B_in} and \eqref{ES:eq:hypothesis_B_out}, 
we find that the loss function is bounded by
\begin{eqnarray}
        \ell(h, \vec{y}, \vec{x}) 
& = &       \| h(\vec{y}) - \vec{x} \|_2 
\leq    \| \vec{x} \|_2+ \| h(\vec{y}) \|_2 \nonumber \\
& \leq &    B_{\operatorname{in}} + B_{\operatorname{out}},
\end{eqnarray}
so that $B_{\operatorname{in}} + B_{\operatorname{out}}$ plays the role of $c$ in Theorem \ref{ES:thm:ge_vs_rademacher}. Besides these boundedness assumptions, we can also 
upper bound the output $f_\mat{\Phi}^l(\mat{Y})$ {with respect to the Frobenius norm} after any number of layers $l$ (in particular for $l <L$, when the layer is not directly followed by an application of the $\sigma$ function) as follows . 
This will be used later in the main technical result, Theorem~\ref{ES:theorem:main_result}.
 
\begin{lemma}\label{ES:lemma:bound_output_after_l_layers}
For any $\mat{\Phi} \in O(N)$, $l \in \N$, and arbitrary $\tau,\lambda > 0$ in $S_{\tau\lambda }$ in the definition \eqref{ES:def_Staulambda} of $f_\mat{\Phi}^l$, we have
\begin{align}
            \left \|f_\mat{\Phi}^l(\mat{Y}) \right \|_F 
&   \leq    \left \| \tau (\mat{A} \mat{\Phi})^\top \mat{Y}  \right \|_F 
\sum_{k=0}^{l-1} \left \|\mat{I} - \tau \mat{\Phi}^\top \mat{A}^\top \mat{A} \mat{\Phi} \right \|_{2\to 2}^{k} 
\label{ES:eq:norm_bound_individual_1} \\ 
&   \leq   
\tau \|\mat{A}\|_{2\to 2} \|\mat{Y}\|_F 
\sum_{k=0}^{l-1} \left \|\mat{I}  - \tau \mat{A}^\top \mat{A} \right \|_{2\to 2}^{k}.
     \label{ES:eq:norm_bound_individual_2}
\end{align}
\end{lemma}

We will encounter the expression 
$\| \mat{I}-\tau \mat{A}^\top\mat{A}\|_{2\to 2}$ more often in the sequel. 
The following remark is useful and shows it can be easily bounded under realistic assumptions.
In particular, we can use it to simplifiy the above estimate to obtain for arbitrary $\mat{\Psi}, \mat{\Phi} \in O(N)$. Namely, under the condition of  $\tau \| \mat{A} \|_{2\to 2}^2 \leq 1$ {and assuming $\vec{y}_i = \mat{A}(\vec{x}_i)$} we have
\begin{align}
        \left\| \mat{\Psi} f_\mat{\Phi}^L(\mat{Y})\right\|_2 
& =       \left\| f_\mat{\Phi}^L(\mat{Y}) \right\|_2 
\leq    L \tau \|\mat{A}\|_{2\to 2} \|\mat{Y}\|_F  {= L \tau \|\mat{A}\|_{2\to 2} \|\mat{A}\mat{X}\|_F} \notag\\
& \leq {L \|\mat{X}\|_F \leq L \sqrt{m} B_{\operatorname{in}}},
\label{ES:eq:linear_growth_L}
\end{align}
i.e., a linear growth with $L$. 
{Note that this is a worst case bound,
and might possibly be improved 
under additional assumptions.}
\begin{remark}\label{ES:rem:tau_and_norm_A}
{Assume $\tau \| \mat{A} \|_{2\to 2}^2 \leq 1$. Then $\| \mat{I}-\tau \mat{A}^\top\mat{A}\|_{2\to 2} \leq 1$.
In the compressive sensing setup ($n < N$), the $N \times N$ - matrix $\mat{A}^\top \mat{A}$ is rank deficient so that even  
$\| \mat{I}-\tau \mat{A}^\top\mat{A}\|_{2\to 2} = 1$ holds in this case.}
\end{remark}

\begin{proof}
{Note that the second inequality \eqref{ES:eq:norm_bound_individual_2} 
immediately follows from \eqref{ES:eq:norm_bound_individual_1}
due to the orthogonality of $\mat{\Phi}$.
We will prove 
\eqref{ES:eq:norm_bound_individual_1} 
via induction.} 
Clearly, for $l = 1$, we have 
$\left \|f_\mat{\Phi}^1(\mat{Y}) \right \|_F = \left \|\tau (\mat{A} \mat{\Phi} )^\top \mat{Y}  \right \|_F$.
Assuming the statement is true for $l$, we obtain it for $l+1$ by the following chain of inequalities, using in particular the contractivity $S_{\tau\lambda}$ with respect to the Frobenius norm,
\begin{eqnarray*}
\left \|f_\mat{\Phi}^{l+1}(\mat{Y}) \right \|_F 
& = &
\left \| S_{\tau\lambda }
    \left[
        \left(\mat{I} - \tau \mat{\Phi}^\top \mat{A}^\top \mat{A} \mat{\Phi} \right ) f^{l}_{\mat{\Phi}} ({\mat{Y}}) 
        + \tau (\mat{A} \mat{\Phi} )^\top \mat{Y} 
    \right] 
\right \|_F \\
& \leq &
\left \|
        \left(\mat{I} - \tau \mat{\Phi}^\top \mat{A}^\top \mat{A} \mat{\Phi} \right ) f^{l}_{\mat{\Phi}} ({\mat{Y}})\|_F 
        + \|\tau (\mat{A} \mat{\Phi} )^\top \mat{Y} 
\right \|_F \\
& \leq &
\left\|\mat{I} - \tau \mat{\Phi}^\top \mat{A}^\top \mat{A} \mat{\Phi}\right \|_{2\to 2}
\left\| f^{l}_{\mat{\Phi}} ({\mat{Y}}) \right \|_F 
+
\left \|\tau (\mat{A} \mat{\Phi} )^\top \mat{Y} \right \|_F \\
& \leq &
\left \|\tau (\mat{A} \mat{\Phi} )^\top \mat{Y} \right \|_F
\left(\sum_{k=0}^{l-1} \left \|\mat{I} - \tau \mat{\Phi}^\top \mat{A}^\top \mat{A} \mat{\Phi} \right \|_{2\to 2}^{k+1} \right)
+
\left \|\tau (\mat{A} \mat{\Phi} )^\top \mat{Y} \right \|_F \\
& = &
\left \|\tau (\mat{A} \mat{\Phi} )^\top \mat{Y} \right \|_F
\sum_{k=0}^{l} \left \|\mat{I} - \tau \mat{\Phi}^\top \mat{A}^\top \mat{A} \mat{\Phi} \right \|_{2\to 2}^{k}.
\end{eqnarray*}
where we have used the induction hypothesis
to arrive at the fourth line.
\end{proof}


\subsection{Dudley's Inequality}
\label{sec:dudley}
We use the following version of Dudley's inequality \citeES[Theorem 8.23]{foucart_mathematical_2013}.
To state the theorem, we require additional definitions. Consider a stochastic process $(X_t)_{t \in \mathcal{T}}$ with the index set $\mathcal{T}$ in a space with pseudo-metric $d$ given by
\[
d(s,t) = \left( \mathbb{E} |X_s - X_t|^2 \right)^{1/2}.
\]
A zero-mean process $X_t$ for $t\in \mathcal{T}$ is called \textit{subgaussian}, if
\[
\mathbb{E} \exp(\theta(X_s-X_t))\leq \exp \left( \theta^2 d(s,t)^2/2 \right)
\quad 
\forall \, s,t\in \mathcal T, \, \theta > 0.
\]
Finally, define the radius of $\mathcal T$ as
$\Delta(\mathcal{T}) = \sup_{t\in \mathcal T} \sqrt{\mathbb{E} |X_t|^2}.$
Dudley's inequality, which will be used to bound the Rademacher complexity term, is stated as follows.

\begin{theorem}[Dudley's inequality]
Let $(X_t)_{t \in \mathcal{T}}$ be a centered (i.e. $\mathbb{E} X_{t} = 0$ for every $t \in \mathcal{T}$) subgaussian process with radius $\Delta(\mathcal T)$. Then
\begin{align}
\mathbb{E} \sup_{t\in\mathcal T} X_t \leq 4\sqrt{2}\int_0^{\Delta(\mathcal T)/2} 
\sqrt{\log\mathcal{N}(\mathcal{T}, d, u)} \, \mathrm{d} u .
\end{align}
\end{theorem}


\subsection{Bounding the Rademacher Complexity}
\label{ES:sec:bound_radem}
Recalling our hypothesis spaces introduced above,
obviously $\mathcal{H}_1^L$ is embedded in $\mathcal{H}_2^L$, i.e., we have the inclusion
\begin{equation}
\mathcal{H}_1^L 
\subset 
\mathcal{H}_2^L.
\end{equation}
For fixed number of layers  $L \in \N$ and $i=1,2$ define the set $\mathcal{M}_i \subset \R^{N \times m}$ as follows:
\begin{align}
\mathcal{M}_i  
= & 
\left\{ 
    \left(h(\vec{y}_1) | \dots | h(\vec{y}_m) \right) \in 
    \R^{N \times m}: h \in \mathcal{H}_i^L
\right\}.
\label{ES:eq:hypSet:1}
\end{align}
For the case $i=2$, the set $\mathcal{M}_2$ corresponding to the hypothesis space $\mathcal{H}_2^L$
reads as 
\begin{equation}
\mathcal{M}_2  
=
\left\{ 
\sigma\left(\mat{\Psi} f_\mat{\Phi}^L (\mat{Y})\right) \in \R^{N \times m}: \mat{\Psi} , \mat{\Phi} \in O(N) 
\right\}.
\label{ES:eq:hypSet:2}
\end{equation}

Note that $\mathcal{M}_2$ is parameterized by $\mat{\Psi} , \mat{\Phi} \in O(N)$ 
(as the hypothesis space $\mathcal{H}_2^L$ is), such that we can rewrite \eqref{ES:rademacher_2} as
\begin{equation}
\mathcal{R}_m (\ell \circ \mathcal{H}_2^L)
   \leq  
   \mathbb{E} \sup_{\mat{M} \in \mathcal{M}_2} 
   \frac{1}{m} \sum_{i=1}^m \sum_{k=1}^N \epsilon_{ik} M_{ik}.
   \label{ES:eq:rademacher_2_rewritten}
\end{equation}
We use Dudley's inequality and a covering number argument to bound the Rademacher complexity term 
The Rademacher process defined in \eqref{ES:eq:rademacher_2_rewritten} is a subgaussian process, and therefore, we can apply Dudley's inequality. 
For the set of matrices $\mathcal{M}_2$ defined above, the radius can be estimated as

\begin{align*}
        \Delta(\mathcal{M}_2) 
 =   &   \sup_{h\in \mathcal{H}_2^L} 
            \sqrt{\mathbb{E}\left(\sum_{i=1}^m \sum_{k=1}^N \epsilon_{ik} h_k(\vec{y}_i)\right)^2} 
\leq      \sup_{h\in \mathcal{H}_2^L} 
            \sqrt{\mathbb{E} \sum_{i=1}^m \sum_{k=1}^N \left(h_k(\vec{y}_i)\right)^2} \\
\leq   &   \sup_{h\in \mathcal{H}_2^L} 
            \sqrt{\sum_{i=1}^m  \left \| h(\vec{y}_i)\right\|^2} 
\leq   \sqrt{m} B_{\operatorname{out}},
\end{align*}
where the last inequality has already been stated in \eqref{ES:eq:output_matrix}.
Plugging this bound in Dudley's inequality, we obtain the following upper bound for the Rademacher complexity,
\begin{equation}
      \mathcal{R}_m (\ell \circ \mathcal{H}_2^L)
\leq  \frac{4\sqrt{2}}{m} 
      \int_{0}^{\sqrt{m} B_{\operatorname{out}}/2}        
\sqrt{\log\mathcal{N}(\mathcal{M}_2,\|\cdot\|_F,\epsilon)} \,\mathrm{d}\epsilon.
\label{ES:eq:dudley_bound}
\end{equation}

We only need to find the covering numbers inside the integral. For that, we bound the covering number of the hypothesis class by the covering number of its parameter space. This is done using a  perturbation analysis argument.  
\subsection{A Perturbation Result}
\label{ABHRESsubsec:no}

The following theorem relates the effect of perturbation of the parameters on the function outputs. This result will be used to bound their covering numbers.

\begin{theorem}\label{ES:thm:perturbation}
Consider the functions $f_{\mat{\Phi}}^L$ defined as in \eqref{ES:eq:thresholding_networks} 
with $L \geq 2$ and dictionary $\mat{\Phi}$ in $O(N)$. 
Then, for any $\mat{\Phi}_1,\mat{\Phi}_2\in O(N)$ we have
\begin{equation}
\left \| f^L_{\mat{\Phi}_1} ({\mat{Y}}) - f_{\mat{\Phi}_2}^L(\mat{Y}) \right \| _F 
\leq K_L
\|\mat{A} \mat{\Phi}_1  - \mat{A} \mat{\Phi}_2 \|_{2 \to 2}
\label{ES:eq:estimate_via_K_L},
\end{equation}
where $K_L$ 
is given by  
\begin{align}
K_L 
& =
\tau \| \mat{Y} \|_F  \|\mat{I} - \tau \mat{A}^\top \mat{A} \|_{2 \to 2}^{L-1} \nonumber \\
& \qquad +
\tau \| \mat{Y} \|_F \sum_{l=2}^{L} 
{
\|\mat{I} - \tau \mat{A}^\top \mat{A} \|_{2 \to 2}^{L-l} 
}
\left(
1+ 2 \tau  \|\mat{A}\|_{2 \to 2}^2  
\sum_{k=0}^{l-2} \|\mat{I} - \tau \mat{A}^\top \mat{A} \|_{2 \to 2}^k
\right).
\label{ES:eq:perturbation_bound}
\end{align}
If $\tau \|\mat{A}\|_{2 \to 2}^2 \leq 1$, we have the simplified upper bound 
\begin{equation}
K_L \leq  \tau \| \mat{Y} \|_F L(L+3).
\label{ES:eq:K_L_simplified}    
\end{equation}
\end{theorem}

{
The bound \eqref{ES:eq:K_L_simplified} follows from  the observation
in Remark \ref{ES:rem:tau_and_norm_A}.
}


\begin{proof} 
We formally set $f^0_{\mat{\Phi}_1} (\mat{Y}) = f^0_{\mat{\Phi}_2}(\mat{Y}) = \mat{Y}$ for a unified treatment of all layers $l \geq 1$. Using the fact that $S_{\tau\lambda }$ is $1$-Lipschitz we obtain
\begin{eqnarray}
&      & \left \| f^l_{\mat{\Phi}_1} ({\mat{Y}}) - f_{\mat{\Phi}_2}^l(\mat{Y}) \right \| _F \nonumber \\
& \leq &  \left \| 
                \left(\mat{I} - \tau ( \mat{A} \mat{\Phi}_1)^\top \mat{A} \mat{\Phi}_1 \right ) f^{l-1}_{\mat{\Phi}_1} ({\mat{Y}}) 
                + \tau (\mat{A} \mat{\Phi}_1 )^\top \mat{Y}  \right . \nonumber \\
             &  & \left. \quad -  
                \left(\mat{I} - \tau ( \mat{A} \mat{\Phi}_2)^\top \mat{A} \mat{\Phi}_2 \right ) f^{l-1}_{\mat{\Phi}_2} ({\mat{Y}}) 
                - \tau (\mat{A} \mat{\Phi}_2 )^\top \mat{Y} 
        \right \| _F \nonumber \\
& \leq &    \left \| 
                \left(\mat{I} - \tau ( \mat{A} \mat{\Phi}_1)^\top \mat{A} \mat{\Phi}_1 \right ) f^{l-1}_{\mat{\Phi}_1} ({\mat{Y}}) 
                -
                \left(\mat{I} - \tau ( \mat{A} \mat{\Phi}_2)^\top \mat{A} \mat{\Phi}_2 \right ) f^{l-1}_{\mat{\Phi}_2} ({\mat{Y}}) 
              \right \| _F  \nonumber \\
                 &  &  \quad + \left \| 
                \tau (\mat{A} \mat{\Phi}_1 )^\top \mat{Y}  
              - \tau (\mat{A} \mat{\Phi}_2 )^\top \mat{Y} 
            \right \| _F \nonumber \\   
& \leq &    \left \| 
                \left(\mat{I} - \tau ( \mat{A} \mat{\Phi}_1)^\top \mat{A} \mat{\Phi}_1 \right ) f^{l-1}_{\mat{\Phi}_1} ({\mat{Y}}) 
                -
                \left(\mat{I} - \tau ( \mat{A} \mat{\Phi}_2)^\top \mat{A} \mat{\Phi}_2 \right ) f^{l-1}_{\mat{\Phi}_2} ({\mat{Y}}) 
              \right \| _F \label{eq:estimate_further}   \\
         & & \quad + 2 \tau \left \|\mat{Y} \right \|_F
        \left \|
            \mat{A} \mat{\Phi}_1 - \mat{A} \mat{\Phi}_2
        \right \|_{2 \to 2}.  \nonumber             
\end{eqnarray}
The term \eqref{eq:estimate_further} is estimated further as follows.
\begin{align*}
   &    \left \| 
                \left(\mat{I} - \tau ( \mat{A} \mat{\Phi}_1)^\top \mat{A} \mat{\Phi}_1 \right ) f^{l-1}_{\mat{\Phi}_1} ({\mat{Y}}) 
                -
                \left(\mat{I} - \tau ( \mat{A} \mat{\Phi}_2)^\top \mat{A} \mat{\Phi}_2 \right ) f^{l-1}_{\mat{\Phi}_2} ({\mat{Y}}) 
        \right \| _F  \\ 
\leq &  \left \| 
                \left(\mat{I} - \tau ( \mat{A} \mat{\Phi}_1)^\top \mat{A} \mat{\Phi}_1 \right ) f^{l-1}_{\mat{\Phi}_1} ({\mat{Y}}) 
                -
                \left(\mat{I} - \tau ( \mat{A} \mat{\Phi}_1)^\top \mat{A} \mat{\Phi}_2 \right ) f^{l-1}_{\mat{\Phi}_1} ({\mat{Y}}) 
        \right. \\
&   \quad    
        \left. +
        \left(\mat{I} - \tau ( \mat{A} \mat{\Phi}_1)^\top \mat{A} \mat{\Phi}_2 \right ) f^{l-1}_{\mat{\Phi}_1} ({\mat{Y}}) 
                -
        \left(\mat{I} - \tau ( \mat{A} \mat{\Phi}_2)^\top \mat{A} \mat{\Phi}_2 \right ) f^{l-1}_{\mat{\Phi}_1} ({\mat{Y}}) 
        \right \| _F  \\
&   \quad     
        \left. +
        \left(\mat{I} - \tau ( \mat{A} \mat{\Phi}_2)^\top \mat{A} \mat{\Phi}_2 \right ) f^{l-1}_{\mat{\Phi}_1} ({\mat{Y}}) 
                -
        \left(\mat{I} - \tau ( \mat{A} \mat{\Phi}_2)^\top \mat{A} \mat{\Phi}_2 \right ) f^{l-1}_{\mat{\Phi}_2} ({\mat{Y}}) 
        \right \| _F   \\       
\leq &  \left \| 
                \left(\mat{I} - \tau ( \mat{A} \mat{\Phi}_1)^\top \mat{A} \mat{\Phi}_1 \right ) f^{l-1}_{\mat{\Phi}_1} ({\mat{Y}}) 
                -
                \left(\mat{I} - \tau ( \mat{A} \mat{\Phi}_1)^\top \mat{A} \mat{\Phi}_2 \right ) f^{l-1}_{\mat{\Phi}_1} ({\mat{Y}}) 
        \right. \\
&   \quad    
        \left. +
        \left(\mat{I} - \tau ( \mat{A} \mat{\Phi}_1)^\top \mat{A} \mat{\Phi}_2 \right ) f^{l-1}_{\mat{\Phi}_1} ({\mat{Y}}) 
                -
        \left(\mat{I} - \tau ( \mat{A} \mat{\Phi}_2)^\top \mat{A} \mat{\Phi}_2 \right ) f^{l-1}_{\mat{\Phi}_1} ({\mat{Y}}) 
        \right. \\
&   \quad     
        \left. +
        \left(\mat{I} - \tau ( \mat{A} \mat{\Phi}_2)^\top \mat{A} \mat{\Phi}_2 \right ) 
        \left( 
        f^{l-1}_{\mat{\Phi}_1} ({\mat{Y}})  
        - 
         f^{l-1}_{\mat{\Phi}_2} ({\mat{Y}}) 
         \right)
        \right \| _F   \\      
\leq &  \left \| 
                  \tau ( \mat{A} \mat{\Phi}_1)^\top \mat{A} \mat{\Phi}_1   f^{l-1}_{\mat{\Phi}_1} ({\mat{Y}}) 
                -
                  \tau ( \mat{A} \mat{\Phi}_1)^\top \mat{A} \mat{\Phi}_2   f^{l-1}_{\mat{\Phi}_1} ({\mat{Y}}) 
        \right. \\
&   \quad    
        \left. +
        \tau ( \mat{A} \mat{\Phi}_1)^\top \mat{A} \mat{\Phi}_2 f^{l-1}_{\mat{\Phi}_1} ({\mat{Y}}) 
                -
        \tau ( \mat{A} \mat{\Phi}_2)^\top \mat{A} \mat{\Phi}_2 f^{l-1}_{\mat{\Phi}_1} ({\mat{Y}}) 
        \right \| _F  \\
&   \quad     
        + \left \|
        \left(\mat{I} - \tau ( \mat{A} \mat{\Phi}_2)^\top \mat{A} \mat{\Phi}_2 \right ) 
        \right \|_{2 \to 2}
        \left \|
        f^{l-1}_{\mat{\Phi}_1} ({\mat{Y}})  
        - 
         f^{l-1}_{\mat{\Phi}_2} ({\mat{Y}}) 
        \right \|_F   \\      
\leq &  
        \left\| \tau ( \mat{A} \mat{\Phi}_1)^\top \right\|_{2 \to 2}
        \left\| 
                (\mat{A} \mat{\Phi}_1 -  \mat{A} \mat{\Phi}_2 )
                f^{l-1}_{\mat{\Phi}_1} ({\mat{Y}}) 
        \right \|_F \\
&   \quad    +
         \tau 
         \left\| 
         (\mat{A} \mat{\Phi}_1)^\top - (\mat{A} \mat{\Phi}_2)^\top 
         \right\|_{2 \to 2}
         \left \|
        \mat{A} \mat{\Phi}_2 f^{l-1}_{\mat{\Phi}_1} ({\mat{Y}}) 
        \right \| _F  \\
&   \quad     
        + \left \|
        \left(\mat{I} - \tau ( \mat{A} \mat{\Phi}_2)^\top \mat{A} \mat{\Phi}_2 \right ) 
        \right \|_{2 \to 2}
        \left \|
        f^{l-1}_{\mat{\Phi}_1} ({\mat{Y}})  
        - 
         f^{l-1}_{\mat{\Phi}_2} ({\mat{Y}}) 
        \right \|_F   \\         
\leq &  
        \tau \left\| \mat{A} \right\|_{2 \to 2}
        \left\| 
                \mat{A} \mat{\Phi}_1 -  \mat{A} \mat{\Phi}_2
        \right \|_{2 \to 2} 
        \left\| f^{l-1}_{\mat{\Phi}_1} ({\mat{Y}}) \right \|_F 
 +
         \tau \left\| \mat{A} \right\|_{2 \to 2}
        \left\| 
                \mat{A} \mat{\Phi}_1 -  \mat{A} \mat{\Phi}_2
        \right \|_{2 \to 2} 
         \left \|f^{l-1}_{\mat{\Phi}_1} ({\mat{Y}}) \right \| _F  \\
&   \quad  
        + \left \|
        \left(\mat{I} - \tau ( \mat{A} \mat{\Phi}_2)^\top \mat{A} \mat{\Phi}_2 \right ) 
        \right \|_{2 \to 2}
        \left \|
        f^{l-1}_{\mat{\Phi}_1} ({\mat{Y}})  
        - 
         f^{l-1}_{\mat{\Phi}_2} ({\mat{Y}}) 
        \right \|_F   \\  
=   &  
        2 \tau \left\| \mat{A} \right\|_{2 \to 2}
                \left\| 
                \mat{A} \mat{\Phi}_1 -  \mat{A} \mat{\Phi}_2
        \right \|_{2 \to 2} 
        \left\| f^{l-1}_{\mat{\Phi}_1} ({\mat{Y}}) \right \|_F 
        + 
        \left \|\mat{I} - \tau \mat{A}^\top \mat{A} \right \|_{2 \to 2}
        \left \|
        f^{l-1}_{\mat{\Phi}_1} ({\mat{Y}})  
        - 
         f^{l-1}_{\mat{\Phi}_2} ({\mat{Y}}) 
        \right \|_F  .
\end{align*}
Plugging this back into \eqref{eq:estimate_further} gives us
\begin{eqnarray}
&      & \left \| f^l_{\mat{\Phi}_1} ({\mat{Y}}) - f_{\mat{\Phi}_2}^l(\mat{Y}) \right \| _F \label{ES:eq:main_estimate_perturbation} \\
& \leq &        
    \left \|\mat{I} - \tau \mat{A}^\top \mat{A} \right \|_{2 \to 2}
    \left \|
        f^{l-1}_{\mat{\Phi}_1} ({\mat{Y}})  
        - 
         f^{l-1}_{\mat{\Phi}_2} ({\mat{Y}}) 
    \right \|_F \nonumber \\
    & & \quad     + \tau \left( 
         2 \left \|\mat{Y} \right \|_F +
         2 \left\| \mat{A} \right\|_{2 \to 2} 
         \left\| f^{l-1}_{\mat{\Phi}_1} ({\mat{Y}}) \right \|_F 
         \right)
        \left \|
            \mat{A} \mat{\Phi}_1 - \mat{A} \mat{\Phi}_2
        \right \|_{2 \to 2} \nonumber  \\        
& \leq &        
    A
    \left \|
        f^{l-1}_{\mat{\Phi}_1} ({\mat{Y}})  
        - 
         f^{l-1}_{\mat{\Phi}_2} ({\mat{Y}}) 
    \right \|_F 
         + B_l
        \left \|
            \mat{A} \mat{\Phi}_1 - \mat{A} \mat{\Phi}_2
        \right \|_{2 \to 2} \label{eq:A_and_B_l}  ,
        \label{ES:eq:estimate_with_A_and_B_l}
\end{eqnarray}
where $A$ and $B_l$ in the previous estimate 
\eqref{ES:eq:estimate_with_A_and_B_l} are given by
\begin{eqnarray*}
&    A & = \left \|\mat{I} - \tau \mat{A}^\top \mat{A} \right \|_{2 \to 2}, \\
&    Z_0 & = 0,   \qquad
     Z_l  = \sum_{k=0}^{l-1} 
      \left \|\mat{I} - \tau \mat{A}^\top \mat{A} \right \|_{2\to 2}^{k}, \quad l \geq 1, \\
&    B_l & = \tau \| \mat{Y} \|_F 
        \left( 2  + 2 \tau \left \| \mat{A} \right \|_{2 \to 2}^2  Z_{l-1}
        \right),   \qquad l \geq 1. 
\end{eqnarray*}

Using these abbreviations, the general formula for $K_L$ in \eqref{ES:eq:perturbation_bound} has the compact form
\begin{equation}
    K_L = \sum_{l=1}^L A^{L-l}B_l, \qquad L \geq 1.
\label{ES:eq:K_L_short}
\end{equation}
Based on \eqref{ES:eq:estimate_with_A_and_B_l} we prove via induction that \eqref{ES:eq:estimate_via_K_L} holds for any 
number of layers $L \in \N$  with $K_L$ given by \eqref{ES:eq:K_L_short}.
For $L = 1$, we can directly calculate the constant $K_1$ via
\begin{align*}
    \left \| f^1_{\mat{\Phi}_1} ({\mat{Y}}) - f_{\mat{\Phi}_2}^1(\mat{Y}) \right \| _F 
=  & \left \| S_{\tau\lambda}(\tau (\mat{A}\mat{\Phi}_1)^\top\mat{Y})-S_{\tau\lambda}(\tau (\mat{A}\mat{\Phi}_2)^\top\mat{Y})\right\| _F \\
\leq & \tau \|\mat{Y}\|_F \left \| \mat{A}\mat{\Phi}_1 - \mat{A}\mat{\Phi}_2   \right \| _{2 \to 2},
\end{align*}
so that $\tau \|\mat{Y}\|_F \leq 2 \tau \|\mat{Y}\|_F = B_1 = K_1$, as claimed in \eqref{ES:eq:K_L_short}.


Now we proceed with the induction step, assuming formula  \eqref{ES:eq:K_L_short} to hold for some $L \in \N$.
Applying the estimate after \eqref{ES:eq:main_estimate_perturbation} 
for the output after layer $L+1$, we obtain
\begin{align*}
\left \| f_{\mat{\Phi}_1}^{L+1} ({\mat{Y}}) - f_{\mat{\Phi}_2}^{L+1}(\mat{Y}) \right \| _F 
\leq &  
A \left \|  f^{L}_{\mat{\Phi}_1} ({\mat{Y}})  - f^{L}_{\mat{\Phi}_2} ({\mat{Y}}) \right \|_F
+
B_{L+1} \left \| \mat{A} \mat{\Phi}_2 - \mat{A} \mat{\Phi}_1 \right \|_{2 \to 2} \\
\leq &  
A K_L   \| \mat{A} \mat{\Phi}_2 - \mat{A} \mat{\Phi}_1  \|_{2 \to 2} 
+
B_{L+1}    \| \mat{A} \mat{\Phi}_2 - \mat{A} \mat{\Phi}_1  \|_{2 \to 2} \\
\leq &  
(A K_L  +  B_{L+1} )  \| \mat{A} \mat{\Phi}_2 - \mat{A} \mat{\Phi}_1  \|_{2 \to 2},
\end{align*}
and therefore,
\begin{eqnarray*}
        K_{L+1} 
& = &   A K_L  +  B_{L+1}
  =     A \sum_{l=1}^L A^{L-l}B_l  + B_{L+1} 
  =     \sum_{l=1}^{L+1} A^{(L+1)-l} B_l.  
\end{eqnarray*}
This is the desired expression for $K_{L+1}$ and finishes the proof of \eqref{ES:eq:estimate_via_K_L}.
It remains to prove the upper bound \eqref{ES:eq:K_L_simplified}.
In Remark \ref{ES:rem:tau_and_norm_A} we have observed that 
$\| \mat{I}-\tau\mat{A}^\top\mat{A} \|_{2\to 2} = 1$ when 
$\tau \|\mat{A}\|_{2\to 2}^2 \leq 1$. Therefore we obtain 
\begin{eqnarray*}
        K_L 
& = &  \sum_{l=1}^L A^{L-l}B_l
\leq  \sum_{l=1}^L B_l
=   \tau \| \mat{Y} \|_F  \sum_{l=1}^L 
        \left( 2  + 2 \tau \left \| \mat{A} \right \|_{2 \to 2}^2  Z_{l-1}
        \right) \\
& \leq &  2 L \tau \| \mat{Y} \|_F    
    +
     2 \tau \| \mat{Y} \|_F  \sum_{l=1}^L Z_{l-1} 
 \leq 2 L \tau \| \mat{Y} \|_F    
        +
     2 \tau \| \mat{Y} \|_F  \sum_{l=1}^L l     \\
& = & 
  \tau \| \mat{Y} \|_F L(L+3),
\end{eqnarray*} 
finishing the proof of the theorem.
\end{proof}

The following result is an adaptation of the previous theorem to take the special form of the final layer into account (a final linear transformation, followed by applying the function $\sigma$).

\begin{corollary}\label{ES:cor:with_Psi} 
Consider the thresholding networks $\mat{\Psi} f_{\mat{\Phi}}^L \in \mathcal{H}_2^L$ as defined in {Section \ref{ES:sec:bound_radem}}, with $L \geq 2$ and $\mat{\Psi}, \mat{\Phi} \in O(N)$. 
Then, for any $\mat{\Phi}_1,\mat{\Phi}_2\in O(N)$ and $\mat{\Psi}_1,\mat{\Psi}_2\in O(N)$ we have
\begin{eqnarray}
& &
\left \| 
\sigma(\mat{\Psi}_1 f^L_{\mat{\Phi}_1} ({\mat{Y}})) - \sigma(\mat{\Psi}_2 f_{\mat{\Phi}_2}^L(\mat{Y}))
\right \| _F \nonumber \\
& \leq  &
M_L \|\mat{\Psi}_1 - \mat{\Psi}_2 \|_{2 \to 2} +
K_L \|\mat{A} \mat{\Phi}_1  - \mat{A} \mat{\Phi}_2 \|_{2 \to 2},
\label{ES:eq:estimate_via_M_L_and_K_L}
\end{eqnarray}
with $K_L$ as in Theorem \ref{ES:thm:perturbation} and 
\begin{equation}
M_L 
= 
\tau \|\mat{A}\|_{2\to 2} \|\mat{Y}\|_F 
\sum_{k=0}^{L-1} \left \|\mat{I}  - \tau \mat{A}^\top \mat{A} \right \|_{2\to 2}^{k} .
\end{equation}
Under the additional assumption that $\tau \|\mat{A}\|_{2 \to 2}^2 \leq 1$
we have
\begin{align*}
& 
\left \| 
\sigma(\mat{\Psi}_1 f^L_{\mat{\Phi}_1} ({\mat{Y}})) - \sigma(\mat{\Psi}_2 f_{\mat{\Phi}_2}^L(\mat{Y}))
\right \| _F \nonumber \\
& \leq  
 \tau \|Y\|_F\left(L \|A\|_{2 \to 2} \| \mat{\Psi}_1 - \mat{\Psi}_2 \|_{2 \to 2} + L(L+3)\|\mat{A} \mat{\Phi}_1  - \mat{A} \mat{\Phi}_2 \|_{2 \to 2}\right). 
\end{align*}
\end{corollary}

\begin{proof} Let us begin with the following estimates, which now include the application of the measurement and the respective dictionary after the final layer. By the $1$-Lipschitzness of $\sigma$, adding mixed terms and applying the triangle inequality, and finally using Theorem \ref{ES:thm:perturbation} for the second summand in the last step we obtain
\begin{eqnarray*}
&    & \left \|
    \sigma \left( \mat{\Psi}_1 f^L_{\mat{\Phi}_1} (\mat{Y}) \right) - 
    \sigma \left( \mat{\Psi}_2 f_{\mat{\Phi}_2}^L (\mat{Y}) \right) 
    \right \|_F \\
&  \leq & 
\left \|\mat{\Psi}_1 f^L_{\mat{\Phi}_1} ({\mat{Y}}) - \mat{\Psi}_2 f^L_{\mat{\Phi}_1} ({\mat{Y}})
+ \mat{\Psi}_2 f^L_{\mat{\Phi}_1} ({\mat{Y}})- \mat{\Psi}_2 f_{\mat{\Phi}_2}^L(\mat{Y}) \right \|_F \\
&  \leq & 
\left\|\mat{\Psi}_1 f^L_{\mat{\Phi}_1} ({\mat{Y}})-\mat{\Psi}_2 f^L_{\mat{\Phi}_1} ({\mat{Y}})\right\|_F +
\left\|\mat{\Psi}_2 f^L_{\mat{\Phi}_1} ({\mat{Y}})-\mat{\Psi}_2 f_{\mat{\Phi}_2}^L(\mat{Y}) \right \|_F  \\ 
&  \leq & 
\left\|f^L_{\mat{\Phi}_1} ({\mat{Y}})\right\|_F \left\|\mat{\Psi}_1 - \mat{\Psi}_2 \right\| _{2 \to 2} +
\left\|f^L_{\mat{\Phi}_1} ({\mat{Y}}) - f_{\mat{\Phi}_2}^L(\mat{Y}) \right \|_F \\
&  \leq & 
\left\|f^L_{\mat{\Phi}_1} ({\mat{Y}})\right\|_F \left\|\mat{\Psi}_1 - \mat{\Psi}_2 \right\| _{2 \to 2} +
K_L \left\|\mat{A} \mat{\Phi}_1  -  \mat{A} \mat{\Phi}_2  \right \|_{2 \to 2}.
\end{eqnarray*}
Now, \eqref{ES:eq:estimate_via_M_L_and_K_L} follows from Lemma \ref{ES:lemma:bound_output_after_l_layers}.
The additional simplified bounds then easily follow from the respective ones in Theorem \ref{ES:thm:perturbation} as well as in \eqref{ES:eq:linear_growth_L}.
\end{proof} 

\begin{remark}
One may try a similar computation like in the proof above for the hypothesis space $\mathcal{H}_1^L$
instead $\mathcal{H}_2^L$. However, after the analog estimate for $\mat{\Phi}_1, \mat{\Phi}_2 \in O(N)$,
\begin{equation*}
\left \|
    \mat{\Phi}_1 f^L_{\mat{\Phi}_1} (\mat{Y}) - \mat{\Phi}_2 f_{\mat{\Phi}_2}^L (\mat{Y}) 
\right \|_F \\
\leq  
\left\|f^L_{\mat{\Phi}_1} ({\mat{Y}})\right\|_F \left\|\mat{\Phi}_1 - \mat{\Phi}_2 \right\| _{2 \to 2} +
K_L \left\|\mat{A} \mat{\Phi}_1  -  \mat{A} \mat{\Phi}_2  \right \|_{2 \to 2},
\end{equation*}
we need to consider both  
$\left\| \mat{A} \mat{\Phi}_1  -  \mat{A} \mat{\Phi}_2  \right \|_{2 \to 2}$
and $\left\| \mat{\Phi}_1 - \mat{\Phi}_2 \right \|_{2 \to 2}$ for later covering number arguments. Using  $\mathcal{H}_2^L$ helps to obtain more concise covering numbers for the class.
Therefore, we decouple the single dictionary applied after the final layer from the previous layers (which all appear together with $\mat{A}$).
\end{remark}

\subsection{Covering number estimates}
\label{ABHRESsubsec:no_2}
{Our proof is built on Dudley's integral in \eqref{ES:eq:dudley_bound}. We need to compute covering numbers $\mathcal{N}\left( \mathcal{M}_2, \|\cdot\|_F, \epsilon \right)$ at different scales $\epsilon > 0$ to evaluate the integral for the space $\mathcal{M}_2$. We start from the following lemma \citeES[Proposition C.3]{foucart_mathematical_2013} and adapt it to our problem. 
}

\begin{lemma}\label{ES:lemma_covering_numbers}
Let $ \epsilon > 0$ and let $\| \cdot \|$ be a norm on a 
$n$-dimensional vector space $V$. 
Then, for any subset $U \subseteq B_{\| \cdot \|} := \{x \in V: \|x\| \leq 1\}$ 
it holds
\[
\mathcal{N} \left( U, \|\cdot\|, \epsilon \right) 
\leq 
\left(1 + \frac{2}{\epsilon} \right)^n.
\]
\end{lemma}






The next lemma provides a bound for product spaces, based on individual covering numbers.
\begin{lemma}\label{ES:lemma:covering_number_product_space}
Consider two metric spaces $(\mathcal{S}_1, d_1), (\mathcal{S}_2, d_2)$.
We define the product metric $\mathcal{S}$, equipped with the metric $d$ by
\begin{equation}
    \mathcal{S} = (\mathcal{S}_1 \times \mathcal{S}_2, d),
    \qquad
    d(x,y) = \sum_{k=1}^2 d_k(x_k,y_k),
\end{equation}
where $x = (x_1, x_2), y = (y_1, y_2) \in \mathcal{S}$.
Then, we have the covering number estimate
\begin{equation}
     \mathcal{N}\left(\mathcal{S}, d, \varepsilon \right)   
\leq \prod_{k=1}^2   
     \mathcal{N}\left(\mathcal{S}_k, d_k, \varepsilon/2 \right).
\end{equation}
\end{lemma}
\begin{proof}
Suppose that, for $k = 1,2$, we have individual coverings of $\mathcal{S}_k$ at level 
$\varepsilon/2$ of cardinality $\mathcal{N}\left(\mathcal{S}_k, d_k , \varepsilon/2 \right)$. 
We will show that the product of all these $\varepsilon/2$-nets is an $\varepsilon$-net for the product space $S$.
Indeed let $x = (x_1, x_2) \in \mathcal{S}$, i.e. $x_k \in \mathcal{S}_k$. Then, for each $x_k \in \mathcal{S}_k$, there exists some element $y_k$ in the $\varepsilon/2$-net of $\mathcal{S}_k$, i.e. $d_k(x_k,y_k) \leq \varepsilon/2$. Then, $y = (y_1,  y_2)$ is an element of the product of all nets, and by the definition of the metric $d$ there is  $d(x,y) \leq \varepsilon/2 + \varepsilon/2 = \varepsilon$.
\end{proof}

The following lemma provides a covering number estimate of $\mat{A}$ applied to the orthogonal group.

\begin{lemma} \label{ES:lem:coverin_A_Phi}
For a fixed matrix $\mat{A} \in \R^{n  \times N }$ consider the set $\mathcal{W}$ defined by
\begin{equation}
\mathcal{W} = \{\mat{A} \mat{\Phi} : \mat{\Phi} \in O(N)\} 
\subset \R^{n \times N},
\label{ES:eq:set_W}
\end{equation} 
i.e., $\mat{A}$ applied to the orthogonal group. The covering number estimate is given by
\[
\mathcal{N} \left(\mathcal{W} , \|\cdot\|_{2 \to 2}, \epsilon \right) 
\leq 
\left(1 + \frac{2 \|\mat{A}\|_{2 \to 2}}{\epsilon} \right)^{nN}.
\]
\end{lemma}

\begin{proof}
First note that $\mathcal{W}$ can be rewritten as
\begin{equation}
\mathcal{W} 
= \left \{ 
    \|\mat{A}\|_{2 \to 2} \frac{\mat{A} \mat{\Phi}}{\|\mat{A}\|_{2 \to 2}} : \mat{\Phi} \in O(N)
  \right \}. 
\label{ES:eq:set_W:2}
\end{equation}
For the covering numbers of the orthogonal group $(O(N ), \|\cdot\|_{2 \to 2})$ 
equipped with the spectral norm we have
\[
\mathcal{N} \left( O(N ), \|\cdot\|_{2 \to 2}, \epsilon \right) 
\leq 
\left(1 + \frac{2}{\epsilon} \right)^{N^2}.
\]
This follows from the fact that the orthogonal group $O(N )$ is contained in $B_{\| \, \cdot \, \|_{2 \to 2}}^{N  \times N }$, 
and therefore Lemma \ref{ES:lemma_covering_numbers} applies. 
This bound then gives
\begin{eqnarray*}
\mathcal{N} \left(\mathcal{W} , \|\,\cdot\,\|_{2 \to 2}, \epsilon \right) 
& = &
\mathcal{N} \left(
\left\{\mat{A} \mat{\Phi} / \|\mat{A}\|_{2 \to 2} : \mat{\Phi} \in O(N) \right \}, 
\|\,\cdot\,\|_{2 \to 2}, \epsilon/\|\mat{A}\|_{2 \to 2} \right)  \\
& \leq & 
\left(1 + \frac{2 \|\mat{A}\|_{2 \to 2}}{\epsilon} \right)^{nN}.
\end{eqnarray*}
\end{proof}


Recall that for Dudleys inequality, we need to estimate the covering numbers
$\mathcal{N} \left( \mathcal{M}_2, \|\, \cdot \,\|_{2 \to 2}, \epsilon \right)$ 
of the set $\mathcal{M}_2$ defined in \eqref{ES:eq:hypSet:2}. In Corollary \ref{ES:cor:with_Psi}, we showed we can estimate 
distances in $\mathcal{M}_2$ via distances of the underlying parameters, 
$\|\mat{\Psi}_1 - \mat{\Psi}_2 \|_{2 \to 2}$ and
$\|\mat{A} \mat{\Phi}_1  - \mat{A} \mat{\Phi}_2 \|_{2 \to 2}$. 
We make use of this in the next corollary, which prepares the application of Dudleys inequality afterwards.

\begin{corollary}
The covering numbers of the set $\mathcal{M}_2$ are bounded by
\begin{align*}
&\log \left(\mathcal{N} \left( \mathcal{M}_2, \|\, \cdot \,\|_{2 \to 2}, \epsilon \right) \right) \\  
&\qquad\leq 
N^2 \cdot \log \left( 1 + \frac{4 M_L}{\epsilon} \right)  
+
nN \cdot \log \left( 1 + \frac{4 \|\mat{A}\|_{2\to 2}K_L}{\epsilon} \right).
\end{align*}
\end{corollary}

\begin{proof}
Using the definition of the set \eqref{ES:eq:set_W}, we have
\begin{eqnarray*}
\mathcal{N} \left(K_L \{\mat{A} \mat{\Phi} : \mat{\Phi} \in O(N)\}, \|\cdot\|_{2 \to 2}, \epsilon \right)
& = & 
\mathcal{N} \left(\{\mat{A} \mat{\Phi}: \mat{\Phi} \in O(N)\}, \|\cdot\|_{2 \to 2}, \epsilon/K_L \right) \\
& \leq &  
\left( 1 + \frac{2\|\mat{A}\|_{2\to 2}K_L}{\epsilon} \right)^{nN}.
\end{eqnarray*}
Furthermore, {since $O(N )\subset B_{\| \, \cdot \, \|_{2 \to 2}}^{N  \times N }$,  
and  by Lemma \ref{ES:lemma_covering_numbers}}
\begin{eqnarray*}
\mathcal{N} \left(M_L \cdot O(N), \|\cdot\|_{2 \to 2}, \epsilon \right) 
& = & 
\mathcal{N} \left( O(N), \|\cdot\|_{2 \to 2}, \epsilon/M_L \right) \\
& \leq &  
\left( 1 + \frac{2 M_L}{\epsilon} \right)^{N^2}.
\end{eqnarray*}
Applying Lemma \ref{ES:lemma:covering_number_product_space} and the previous estimates (with $\varepsilon/2$ instead of $\varepsilon$), we can now estimate the covering number of $\mathcal{M}_2$
\begin{align*}
            \mathcal{N} \left( \mathcal{M}_2, \|\, \cdot \,\|_F, \epsilon \right) 
& \leq     \mathcal{N}
            \left(
            M_L \cdot O(N) \times K_L \cdot \mathcal{W}, \|\, \cdot \,\|_{2 \to 2}, \epsilon 
            \right) \\
& \leq     \mathcal{N} \left(M_L \cdot O(N), \|\, \cdot \,\|_{2 \to 2}, \epsilon/2 \right)             
            \mathcal{N} \left(K_L \cdot \mathcal{W}, \|\, \cdot \,\|_{2 \to 2}, \epsilon/2 \right)  \\
& \leq     \left( 1 + \frac{4M_L}{\epsilon} \right)^{N^2}
            \left( 1 + \frac{4\|\mat{A}\|_{2\to 2}K_L}{\epsilon} \right)^{nN},
\end{align*}
which immediately gives us the desired statement.
\end{proof}

\subsection{Main result}
\label{ABHRES:subsec:mainresult}

Finally, we are able to state and prove the main result of this section.

\begin{theorem}\label{ES:theorem:main_result}
Consider the hypothesis space $\mathcal{H}_2^L$ defined in \eqref{ES:eq:hypothesis_class:H2}. 
With probability at least $1-\delta$, for all $h \in \mathcal{H}_2^L$, the generalization error  is bounded as
\begin{eqnarray*}
\mathcal{L}(h) 
& \leq &
\hat{\mathcal{L}}(h) 
+ 8 B_{\operatorname{out}} \sqrt{\frac{Nn}{m}}
  \sqrt{\log{e\left(1 +\frac{8 K_L\|\mat{A}\|_{2\to 2}}{\sqrt{m} B_{\operatorname{out}}}\right)}} \\
&  & \quad 
+ 8 B_{\operatorname{out}} \frac{N}{\sqrt{m}}
    \sqrt{\log{e\left(1 +\frac{8 M_L}{\sqrt{m} B_{\operatorname{out}}  }\right)}} 
+ 4 (B_{\operatorname{in}} + B_{\operatorname{out}}) \sqrt{\frac{2\log(4/\delta)}{m}},
\label{ES:eq:main_thm_generalization_bound}
\end{eqnarray*}
where $K_L$ is the {constant} in \eqref{ES:eq:perturbation_bound}. 
\end{theorem}

\begin{proof}
For the proof it remains to bound the Rademacher complexity via Dudley's integral \eqref{ES:eq:dudley_bound},
for which in turn we use the covering number arguments from the previous subsection as follows,
\begin{eqnarray*}
       \mathcal{R}_m (\ell \circ \mathcal{H}_2^L) 
& =& \mathbb{E} \sup_{\mat{M} \in \mathcal{M}_2} \frac{1}{m} 
        \sum_{i=1}^m \sum_{k=1}^N \epsilon_{ik} M_{ik} \\
& \leq &   \frac{4\sqrt 2}{m}\int_0^{\sqrt{m} B_{\operatorname{out}}/2}  
        \sqrt{\log\mathcal{N}\left(\mathcal{M}_2, , \|\, \cdot \,\|_F, \epsilon \right)} 
        \;\mathrm{d} \epsilon \\
& \leq &   \frac{4\sqrt 2}{m}\int_0^{\sqrt{m} B_{\operatorname{out}}/2}  
        \sqrt{N^2 \cdot \log \left( 1 + \frac{4 M_L}{\epsilon} \right)} 
        \;\mathrm{d} \epsilon     \\    
&  & \quad +  \frac{4\sqrt 2}{m}\int_0^{\sqrt{m} B_{\operatorname{out}}/2}  
        \sqrt{nN \cdot \log \left( 1 + \frac{4 \|\mat{A}\|_{2\to 2}K_L}{\epsilon} \right)} 
        \;\mathrm{d} \epsilon     \\         
& \leq &   \frac{4\sqrt 2 N}{m}\int_0^{\sqrt{m} B_{\operatorname{out}}/2}  
        \sqrt{\log \left( 1 + \frac{4 M_L}{\epsilon} \right)} 
        \;\mathrm{d} \epsilon     \\    
&  & \quad +  \frac{4 \sqrt{2nN}}{m}\int_0^{\sqrt{m} B_{\operatorname{out}}/2}  
        \sqrt{\log \left( 1 + \frac{4 \|\mat{A}\|_{2\to 2}K_L}{\epsilon} \right)} 
        \;\mathrm{d} \epsilon     \\ 
& \leq &  2 \sqrt{2} B_{\operatorname{out}} \frac{N}{\sqrt{m}}
          \sqrt{\log{\left(e\left(1 +\frac{4 M_L}{\sqrt{m} B_{\operatorname{out}}/2  }\right)\right)}} \\
&  & \quad +  2 \sqrt{2} B_{\operatorname{out}} \sqrt{\frac{Nn}{m} }
    \sqrt{\log{\left(e\left(1 +\frac{4 K_L\|\mat{A}\|_{2\to 2}}{\sqrt{m} B_{\operatorname{out}}/2}\right)\right)}}.
\end{eqnarray*}
where we have used the following inequality for the last step \citeES[Lemma C.9]{foucart_mathematical_2013}
\begin{equation}
\int_0^\alpha \sqrt{\log \left(1 + \frac{\beta}{t}\right)} \;\mathrm{d} t
\leq 
\alpha\sqrt{ \log \left(e(1+\beta/\alpha)\right) } 
\quad\text{for}\quad 
\alpha,\beta>0.
\end{equation}
The theorem is obtained using Theorem \ref{ES:thm:ge_vs_rademacher} with the upper bound 
$c =  B_{\operatorname{in}} + B_{\operatorname{out}}$ 
for the functions output from {\eqref{ES:eq:output_matrix}}, 
and bounding the Rademacher complexity term \eqref{ES:rademacher_2} 
with the generalized contraction principle Lemma \ref{ES:lem:Maurer},
which in turn is bounded using Dudleys integral as above. 
\end{proof}

Let us make the reasonable assumption that $\tau \|\mat{A}\|_{2 \to 2} \leq 1$.
Taking into account that $M_L\leq \tau \|\mat{A}\|_{2 \to 2} \|\mat{Y}\|_F L$, see also 
\eqref{ES:eq:linear_growth_L}, i.e., $M_L$ scales at most linearly in $L$ (which remains inside the logarithm),
and since $K_L$ depends quadratically on $L$, see \eqref{ES:eq:K_L_simplified}, we have
\begin{equation}
\mathcal{L}(h) - \hat{\mathcal{L}}(h) \lesssim
 \frac{N}{\sqrt{m}} \sqrt{\log(L)} + \sqrt{\frac{Nn}{m}}  \sqrt{\log(L)}
 \sim \sqrt{\frac{\log(L) N(N+n)}{m}} \sim \sqrt{\frac{\log(L) N^2}{m}},
\end{equation}
where the last relation holds under the reasonable assumption that $1 \leq n \leq N$.
This estimate is stated more rigorously in the following corollary.



\begin{corollary}\label{ES:corollary:main_result}
Consider the hypothesis space $\mathcal{H}_2^L$ defined in \eqref{ES:eq:hypothesis_class:H2} {and assume that $\tau \|\mat{A}\|_{2 \to 2}^2 \leq 1$.}
With probability at least $1-\delta$, for all $h \in \mathcal{H}_2^L$, the generalization error 
is bounded as
\begin{eqnarray*}
\mathcal{L}(h) 
& \leq &
\hat{\mathcal{L}}(h) 
+ 8 B_{\operatorname{out}} \sqrt{\frac{Nn}{m}}
\sqrt{1 + \log \left(2 + \frac{8 L(L+3) \tau \|\mat{Y}\|_F \|\mat{A}\|_{2 \to 2}}{\sqrt{m}B_{\operatorname{out}}}\right)} \\
&  & \quad 
+ 8 B_{\operatorname{out}} \frac{N}{\sqrt{m}}
\sqrt{\log{e\left(1 +
\frac{8 \tau L\|\mat{A}\|_{2\to 2} \|\mat{Y}\|_F}{\sqrt{m} B_{\operatorname{out}}  }\right)}
} 
+ 4 (B_{\operatorname{in}} + B_{\operatorname{out}}) \sqrt{\frac{2\log(4/\delta)}{m}},
\label{eq:ge_vs_rademacher:2}
\end{eqnarray*}
where $K_L$ is the perturbation bound in \eqref{ES:eq:perturbation_bound}. 
\end{corollary}

{Theorem \ref{ABHRES:thm:main_result} follows immediately from Theorem \ref{ES:theorem:main_result} and Corollary \ref{ES:corollary:main_result}. Indeed, we have seen in \eqref{ES:eq:linear_growth_L} that
\begin{equation}
    \tau \|\mat{A}\|_{2\to 2} \|\mat{Y}\|_F\leq \sqrt{m}B_{\operatorname{in}}.
\end{equation}

 

Using that $B_{\operatorname{in}}= B_{\operatorname{out}}$ by assumption,
and, for simplicity also assuming that $L \geq 2$ such that 
$2 + 8 L(L+3) \leq (5L)^2$,
we therefore have
\[
\log \left(2 + \frac{8 L(L+3) \tau \|\mat{Y}\|_F \|\mat{A}\|_{2 \to 2}}{\sqrt{m}B_{\operatorname{out}}}\right) 
\leq 
\log(2 + 8 L(L+3)) \leq 2 \log(5L).
\]

Plugging in these estimates and using that $\mathcal{H}_1^L\subseteq\mathcal{H}_2^L$ gives the statement of Theorem~\ref{ABHRES:thm:main_result}.}


As already pointed out above, the deep thresholding network we analyse is, due to the weight sharing, a recurrent neural network.
The authors of  \citeES{dasgupta_sample_1996} derive VC-dimension bounds of recurrent networks for recurrent perceptrons with binary outputs. The VC-dimension of recurrent neural networks for different classes of activation functions has been studied by the authors of  \citeES{koiran_vapnik-chervonenkis_1998}. However, their results do not immediately apply to our setup, since they focus on one-dimensional inputs and outputs, which of course does not suit our vector-valued regression problem, and moreover, would correspond to taking just one single measurement.
Even in the scenario which is closest to ours, namely fixed piecewise polynomial activation functions with $n=1$, their VC dimension bound scales between $\mathcal{O}(Lw)$ and $\mathcal{O}(Lw^2)$, where $L$ is the number of layers and $w$ is the number of trainable parameters in the network. In our case, the number of trainable parameters is equal to the dimension of the orthogonal group $O(N)$, which is $N (N - 1) / 2$. Therefore, their bounds scale between $\mathcal{O}(L N^2)$ and $\mathcal{O}(L N^4)$.
In contrast, if $n = 1$, our bound scales only like $\mathcal{O}(N\sqrt{\log(L)})$.
Besides, we only make use of Lipschitzness of the activation function.
\section{Thresholding Networks for Sparse Recovery}
\label{ABHRESsec:5}

{
In Theorem \ref{ABHRES:thm:main_result} and Corollary \ref{ES:corollary:main_result}, we have provided a worst-case bound on the sample complexity that holds uniformly over the hypothesis space and for any arbitrary data distribution. It is interesting to see if this bound can be improved for data distributions limited to low complexity sets distributions, for example over the set of sparse vectors. ISTA is used mainly in sparse coding and recovery tasks, therefore it is reasonable to ask if the generalization error behaves similarly when it is applied to sparse recovery tasks. 
}


We consider  a synthetic dataset as well as the MNIST data set \citeES{lecun1998mnist}. For both cases, the measurement matrix is a random Gaussian matrix properly normalized to guarantee convergence of soft-thresholding algorithms. The synthetic data is generated for different input and output dimensions and sparsity level. The original dictionary is a random orthogonal matrix. The default parameters are $N=120$, $n=80$ and sparsity equal to $10$. Sparse vectors are generated by choosing their support uniformly randomly and then picking non-zero values according to the standard normal distribution. The experiments for the synthetic data are repeated at least 50 times, and the results are averaged over the repetitions. For both the MNIST and the synthetic dataset, we sweep over $L, N$ and $n$ to see how the generalization error behaves.    

There are different ways to implement the orthogonality constraint for weight matrices. One way  \citeES{lezcano2019cheap} is based on the fact that the matrix exponential mapping provides a bijective mapping from the skew-symmetric matrices onto the special orthogonal group $SO(N)$. However, we use the alternative method of adding a regularization term  $\|\vec{I} - \mat{\Phi}^\top \mat{\Phi}\|_F$ (or another matrix norm) to the loss function, which means to penalize $\mat{\Phi}$ that is far from being orthogonal. 

We choose different number of measurements  and layers  for both datasets. For each one, the network is trained for a few epochs. Mostly not more than 10 epochs are required to get first promising results, and often times, the loss goes down very slowly after 10 epochs. 

All experiments (see Figure \ref{fig:abserror_MNIST}) show that it is possible to recover the original vectors $\vec{x}$ with as few as 10 layers, which is less than typical when using ISTA (see supplementary materials for some visuals). Note that the error in the MNIST experiments is the pixel-based error normalized by the image dimension and MNIST pixels are all normalized between $0$ and $1$.
We have chosen {ISTA} with a similar structure and 5000 iterations.  The result warrants the applicability of dictionary learning for sparse reconstruction.
\begin{figure}[th]
    \centering
    \begin{subfigure}[b]{0.47\textwidth}
\includegraphics[width=\textwidth]{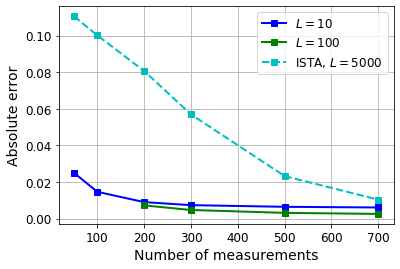}
        \caption{Absolute reconstruction error for different measurements of MNIST}
        \label{fig:abserror_MNIST}
    \end{subfigure}
\centering
    \begin{subfigure}[b]{0.47\textwidth}
\includegraphics[width=\textwidth]{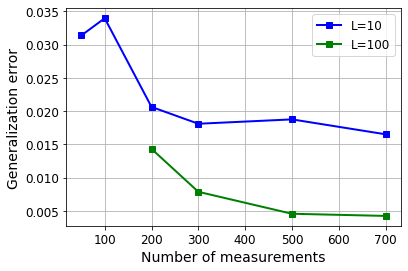}
        \caption{Generalization error for different measurements of MNIST}
        \label{fig:generror_MNIST}
    \end{subfigure}
\caption{MNIST dataset}
\label{fig:mnist_dataset}
\end{figure}

Figure \ref{fig:generror_synth_n} confirms the dependence of the generalization error on the number of layers $L$.
Increasing the number of 
layers increases the generalization error for a fixed number of measurements $n$.
However,  the generalization error decreases by increasing the number of layers for MNIST dataset.
For both synthetic and the MNIST dataset, it seems that increasing the number of measurements decreases the generalization error.  
See Figure \ref{fig:generror_MNIST}, \ref{fig:generror_synth_n} and \ref{fig:generror_synth_N}. Besides, Figure \ref{fig:generror_synth_N} shows that increasing $N$ increases the generalization error. Therefore, our bound scales correctly with the input dimension and the number of layers but incorrectly with the number of measurements. Although not predicted by our theoretical results, this is not unexpected. Note that the number of measurements $n$ is not essential here, since it can always be upper bounded by $N$. Therefore, the theoretical bound on the generalization error (see \eqref{eq:general_error_rough}, and Theorem \ref{ABHRES:thm:main_result} as well as Corollary \ref{ES:corollary:main_result} for more details) can be lower and upper bounded via
\[
\sqrt{\frac{\log(L)}{m}} N \leq \sqrt{\frac{\log(L)}{m}} (N + \sqrt{Nn}) \leq 2\sqrt{\frac{\log(L)}{m}} N.
\]
Furthermore, as mentioned above, the sample complexity is supposed to apply to all possible input distributions, If we restrict ourselves to distributions over low complexity sets, then various worst-case bounds in our analysis might be improved. The experiments seem to confirm this intuition. Namely, for the MNIST dataset there is a clear improvement with increasing  the number of measurements and the number of layers. This is intuitive from a compressive sensing standpoint, as more number of layers in ISTA leads to better results and more measurements provide more information about the input.

On the other hand, the synthetic dataset shows that the generalization error increases with the input dimension and the number of layers. Note that the bound of this chapter is obtained for a very general setting where nothing is assumed on the data structure. Additional assumptions on the structure of the problem, i.e., sparsity can be used to improve the current bound. Nonetheless, the linear dimension dependency of the current bound makes it a very good baseline for future comparisons.

\begin{figure}[h]
    \centering
    \begin{subfigure}[b]{0.47\textwidth}
\includegraphics[width=\textwidth]{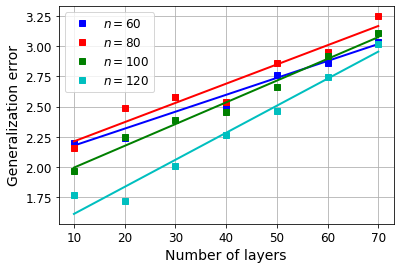}
        \caption{Generalization error for different measurements of synthetic data ($N=120$)}
        \label{fig:generror_synth_n}
    \end{subfigure}
\centering
    \begin{subfigure}[b]{0.47\textwidth}
\includegraphics[width=\textwidth]{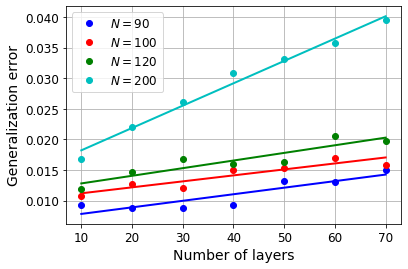}
        \caption{Generalization error for different input dimensions of synthetic data ($n=80$)}
        \label{fig:generror_synth_N}
    \end{subfigure}
\caption{Generalization error for synthetic dataset}
\label{fig:synth_dataset}
\end{figure}


{The model that is used for our experiments shares the weights across layers conforming to our theoretical setup. However, we can improve the performance of this method by using ideas similar to LISTA literature. Many works on LISTA use a different dictionary at each layer, which eases the training procedure and can lead to potentially better results.}


{
\section{Conclusion and Outlook}
}

In this chapter, we have derived a generalization bound for an unfolded ISTA algorithm where, similar to LISTA, {the dictionary is  learned via learning the reconstruction algorithm, interpreted as neural network with shared layers.} 
To the best of our knowledge, this is the first result of its kind. Our proof utilizes a Rademacher complexity analysis and obtains generalization bounds with only linear dependence on the dimension.
The comparison of our theoretical results and the numerical results suggests that we might be able to obtain tighter generalization bounds of neural networks for structured input data. Future works consist of also considering more intricate structures with more flexible weight sharing between the layers and also learning parameters {such as the stepsizes and thresholds simultaneously}.

\bigskip


\begin{acknowledgement}
The authors would like to thank Sebastian Lubjuhn for proofreading an earlier version of this paper and giving valuable suggestions for improvement.
The third author acknowledges funding from the Deutsche Forschungsgemeinschaft (DFG) through the project \emph{Structured Compressive Sensing via Neural Network Learning} (SCoSNeL, MA 1184/36-1) within the SPP 1798 \emph{Compressed Sensing in Information Processing} (CoSIP).
\end{acknowledgement}

\bibliographystyleES{spmpsci}                 
\bibliographyES{bibfileES}                


	
\end{document}